\documentclass[sn-mathphys-num]{sn-jnl}

\usepackage{graphicx}%
\usepackage{multirow}%
\usepackage{amsmath,amssymb,amsfonts}%
\usepackage{amsthm}%
\usepackage{mathrsfs}%
\usepackage[title]{appendix}%
\usepackage{xcolor}%
\usepackage{textcomp}%
\usepackage{manyfoot}%
\usepackage{booktabs}%
\usepackage{algorithm}%
\usepackage{algorithmicx}%
\usepackage{algpseudocode}%
\usepackage{listings}%

\newcommand{\Rmo}{\mathbb{R}^{m+1}}
\newcommand{\C}{\mathbb{C}}

\newcommand{\R}{\mathbb{R}}
\newcommand{\Clm}{\mathcal{C}l_m}
\newcommand{\CL}{\mathcal{L}}

\newcommand{\Su}{\mathbb{S}}
\newcommand{\be}{\begin{eqnarray*}}
\newcommand{\ee}{\end{eqnarray*}}
\newcommand{\ux}{\underline{\bx}}
\newcommand{\Real}{\text{Re}}

\newcommand{\bo}{\boldsymbol}
\newcommand{\bx}{\boldsymbol{x}}

\newcommand{\bq}{\boldsymbol{q}}

\newcommand{\boe}{\boldsymbol{e}}
\newcommand{\ubq}{\underline{\bq}}
\newcommand{\ubx}{\underline{\bx}}
\newcommand{\bv}{\bigg\vert}
\allowdisplaybreaks

\theoremstyle{thmstyleone}%
\newtheorem{theorem}{Theorem}[section]
\newtheorem{corollary}[theorem]{Corollary}
\newtheorem{lemma}[theorem]{Lemma}
\newtheorem{proposition}[theorem]{Proposition}

\theoremstyle{thmstyletwo}%
\newtheorem{remark}{Remark}%

\theoremstyle{thmstylethree}%
\newtheorem{definition}{Definition}%

\begin{document}

\title[Teodorescu transform for slice monogenic functions and applications]{Teodorescu transform for slice monogenic functions and applications}

\author*[1]{\fnm{Chao} \sur{Ding}}\email{cding@ahu.edu.cn}

\author[2]{\fnm{Zhenghua} \sur{Xu}}\email{zhxu@hfut.edu.cn}

\affil*[1]{\orgdiv{Center for Pure Mathematics, School of Mathematical Sciences}, \orgname{Anhui University}, \orgaddress{\street{Jiulong Road 111}, \city{Hefei}, \postcode{230601}, \state{Anhui}, \country{China}}}

\affil[2]{\orgdiv{School of Mathematics}, \orgname{Hefei University of Technology}, \orgaddress{\street{Danxia Road 485}, \city{Hefei}, \postcode{230601}, \state{Anhui}, \country{China}}}

\abstract{In the past few years, the theory of slice monogenic functions has been fully developed, especially in its applications to the noncommutative $S$-functional calculus. In this article, we introduce the Teodorescu transform in the theory of slice monogenic functions, which turns out to be a right inverse of the slice Cauchy-Riemann operator. The boundednesses of the Teodorescu transform and its derivatives are investigated as well. These results successfully generalize the classical results in the complex plane to higher dimensions. As applications, a Hodge decomposition of the Banach space of slice $L^p$ functions and the corresponding generalized Bergman projection are investigated.}



\maketitle
\section{Introduction}
The importance of complex analysis for mathematical physics is that the Cauchy-Riemann operator and its complex conjugate operator provide a factorization for the two-dimensional Laplace operator. The Laplace operator arises from various applications in many partial differential equations, for instance, Poisson equation, wave equation, heat equation, etc. In the case of the in-homogeneous Cauchy-Riemann equation, it leads to the well-known Cauchy-Pompeiu integral representation formula. In particular, the area integral in the complex Cauchy-Pompeiu integral representation defines a singular integral operator, which is usually called the Teodorescu transform. This transform is considered as the right inverse of the Cauchy-Riemann operator.
\par
The Teodorescu transform plays an important role in the Vekua theory in solving certain first order partial differential equations, such as the Beltrami equation, see \cite{Vekua}. Many researchers have contributed to the study of this topic in the context of Clifford analysis. For instance, Begehr and Hile \cite{B1,B2}, Wen and Begehr \cite{B3}, and Bojarski \cite{Bo} studied the properties of the Teodorescu transform and a Ahlfors-Beurling $\Pi$-operator for solving complex first order partial differential equations. K\"ahler \cite{Kah} studied the Teodorescu transform and the Beltrami equation in the quaternionic case. G\"urlebeck et al. \cite{Gur} investigated a class of generalized complex $\Pi$-operator in hyperholomorphic function theory. A systematic study of the Teodorescu transform in classical Clifford analysis can be found in \cite{23}. The Teodorescu transform in the octonionic space was investigated  by Wang and Bian \cite{Haiyan} in 2017. More work has been done on $\Pi$ operator via the Teodorescu transform, cf. \cite{GK,PU}.

In 2006, Gentili and Struppa introduced the concept of slice regular functions over quaternions \cite{19,20}, which was inspired by an earlier work done by Cullen in \cite{9}. This work attracted the attention of many researchers to a systematic investigation of the theory of slice regular functions. In 2009, Colombo, Sabadini and Struppa \cite{6} generalized this idea to the general higher dimensions with the concept of slice monogenic functions. In \cite{6}, slice monogenic functions are defined as functions which are holomorphic on each slice in the Euclidean space, and the theory of slice monogenic functions has been well-developed so far, see, e.g. \cite{7}.
\par
The theory of slice regular and monogenic functions also has deep motivation in quantum physics. Back in 1930s, with the paper of Birkhoff and von Neumann on the logic of quantum mechanics, quantum mechanics can be formulated over the real, the complex, and the quaternionic numbers. The main problem is that the precise definition of the spectrum of  quaternionic linear operators is unclear. In 2006, the discovery of $S$-spectrum by Colombo and Sabadini in the theory of slice hyperholomorphic functions became the underlining function theory on which new functional calculi for quaternionic operators and for $n$-tuples of operators have been developed. Properties and formulations of the quaternionic functional calculus were investigated by the same authors, see \cite{CS1,CS2}. Since then, the quaternionic spectral theory based on the notion of $S$-spectrum has been well-developed, and much work has contributed to this topic, see e.g. \cite{ACS,CGK,7,CG,GMP,GMP1}.
\par
In 2011, Ghiloni and Perotti \cite{15} investigated the theory of slice regular functions on real alternative algebras with a different approach by introducing concepts of stem functions and slice functions. The theory of slice regular functions has been developed even further with this approach. For instance, the theory of slice regular functions has been developed into slice Fueter-regular functions over octonions \cite{14, 26}. More details on slice  regular and monogenic functions can be found, for instance, in \cite{7,13}.
\par
In 2013, Colombo et al. \cite{4} introduced a non-constant coefficients differential operator, whose null solutions are closely related to slice monogenic functions when the domains are given with suitable conditions.
Later on, many researchers started to investigate this global differential operator for slice regularity. For instance, in \cite{18}, Ghiloni and Perotti introduced a volume Cauchy integral formula and a volume Borel-Pompeiu formula for slice regular functions on real associative *-algebras. In \cite{22}, a global Borel-Pompeiu formula and a global Cauchy-type formula for the non-constant coefficients differential operator were presented for quaternionic slice regular functions. In \cite{11}, the authors introduced a Borel-Pompeiu formula and a slice regular extension result for the slice Cauchy-Riemann operator over Clifford algebras, which provides necessary tools for the investigation done in this article.  In this paper, we shall use the Cauchy kernel in the slice monogenic setting to define a Teodorescu transform, which turns out to be a right inverse of the slice Cauchy-Riemann operator. Some other properties of this transform will be studied as well.
\par
This paper is organized as follows. In Section 2, we will review some preliminaries on slice monogenic functions. Section 3 is devoted to an investigation into boundednesses of the Teodorescu transform and its derivatives. The fact that the Teodorescu transform is a right inverse of the slice Cauchy-Riemann operator is also studied here. In Section 4, a Hodge decomposition of the Banach space of slice $L^p$ functions is introduced, which gives rise to a generalized Bergman projection in the theory of slice monogenic functions.
\section{Preliminaries}
In this section, we review some definitions and preliminary results on slice Clifford analysis, for more details, we refer the readers to \cite{7,13,15}.
\par
Let $\mathfrak{B}=\{\bo{e}_1,\ldots,\bo{e}_m\}$ be a standard orthonormal basis of the $m$-dimensional Euclidean space $\R^m$. The real Clifford algebra $\Clm$ is generated by $\mathfrak{B}$ by considering the relationship
$$\bo{e}_i\bo{e}_j+\bo{e}_j\bo{e}_i=-2\delta_{ij},$$
where $\delta_{ij}$ is the Kronecker delta function. Hence, each Clifford number $x\in \Clm$ can be written as $x=\sum_{A}x_Ae_A$ with real coefficients $x_A\in \mathbb{R}$ and $A\subset\{1,\ldots,m\}$. It also implies that $\Clm$ can be considered as a $2^m$ dimensional vector space with a basis $\{e_A\},\ A\subset\{1,\ldots,m\}$. We introduce a norm of $x=\sum_{A}x_Ae_A$ as $|x|=(\sum_{A}x_A^2)^{\frac{1}{2}}.$
 If we denote $\Clm^{k}=\{x\in \Clm:\ x=\sum_{|A|=k}x_Ae_A\}$, where $|A|$ stands for the cardinality of the set $A$, then one obtains a decomposition $\Clm=\bigoplus_{k=0}^m \Clm^k$. In particular, we call elements in $\Clm^0$ scalars and elements in $\Clm^1$ $1$-vectors. It is easy to see that the $(m+1)$-dimensional Euclidean space $\R^{m+1}=\R\oplus\R^m$ can be identified with $\Clm^0\oplus \Clm^1$, whose elements are also called paravectors, and $\R$ is identified with the vector subspace of $\Clm$ generated by $\bo{e}_{\emptyset}=1$.
 \par
 The Clifford conjugation of $ x=\sum_{A}x_Ae_A \in \Clm$ is defined by
\begin{align*}
\overline{x}=\sum_{A}(-1)^{\frac{|A|(|A|+1)}{2}}x_Ae_A.
\end{align*}
\par
In this article, a $1$-vector is denoted by $\ubx=\sum_{k=1}^mx_k\boe_k$. Then for a paravector $\bx\notin\R$, we can write it as $\bx=x_0+\ubx=:\Real[\bx]+\displaystyle\frac{\ubx}{|\ubx|}|\ubx|=:u+I_{\bx}v$, where $ u=x_0,\ v=|\ubx|=\bigg(\displaystyle\sum_{j=1}^mx_j^2\bigg)^{\frac{1}{2}}$ and $I_{\bx}=\displaystyle\frac{\ubx}{|\ubx|}$. When $\bx\in \R$, $I_{\bx}$ can be chosen to be an arbitrary unit vector in the $(m-1)$-unit sphere $\Su$, where $\Su$ is defined by
\be
\Su:=\{\ubx=\boe_1x_1+\cdots+\boe_mx_m\in\Clm^1:\ x_1^2+\cdots +x_m^2=1\}.
\ee
Given an $I\in\Su$, let $\C_I$ be the plane spanned by $1$ and $I$, which is isomorphic to the complex plane. Hence, arbitrary element $\bx\in \C_I$ can be denoted by $\bx=u+Iv$ with $u,v\in\R$. Recall that any paravector $\bx$ can be written as $\bx=u+I_{\bx}v$, this suggests that we can write it in a form as an element in some complex plane $\C_I$.
\par
Given a paravector $\bo{x}=x_0+\underline{\bo{x}}=x_0+I_{\bo{x}}|\underline{\bo{x}}|\in\Rmo$, we introduce a set
\be
[\bo{x}]=\{x_0+I|\underline{\bo{x}}|\ :\ I\in\Su\}.
\ee
Indeed, one can easily see that $[\bo{x}]$ is either a point (when $\bo{x}\in\R$) or the $(m-1)$-sphere with center at $x_0$ and radius $|\underline{\bo{x}}|$.
\begin{definition}
Let $\Omega\subset\Rmo$ be a domain and $f:\ \Omega\longrightarrow \Clm$ be a real differentiable function. Given any $I\in\Su$, we denote by $f_I$ the restriction of $f$ to $\Omega_I:=\Omega\cap\C_I$. The function $f$ is called \emph{left slice monogenic}, if for all $I\in\Su$, we have
$
\frac{1}{2}\big(\frac{\partial}{\partial u}+I\frac{\partial}{\partial v}\big)f_I(u+Iv)=0
$
on $\Omega_I$. Due to the noncommutativity of multiplication of Clifford numbers, we can also define \emph{right $s$-monogenicity} by requiring that $
f_I(u+Iv)\frac{1}{2}\big(\frac{\partial}{\partial u}+I\frac{\partial}{\partial v}\big)=0
$
on $\Omega_I$ for all $I\in\Su$.
\end{definition}
Next, we review the approach to slice regularity with the concept of \emph{stem functions}, which was introduced in \cite{15}.
\begin{definition}
Let $D\subset\mathbb{C}$ be invariant with respect to complex conjugation, i.e., if $z=u+iv\in D$, then its complex conjugation $\overline{z}=u-iv\in D$ as well.  A function $F:\ D\longrightarrow \Clm\otimes\C$ satisfying $F(\overline{z})=\overline{F(z)}$ for all $z\in D$ is called a \emph{stem function} on $D$.
\end{definition}
Notice that a complex Clifford-valued stem function $F(z)\in \Clm\otimes\C$ can be written as $F(z)=F_1(z)+iF_2(z)$, where $F_1(z), F_2(z)\in \Clm$. Then, $F_1$ and $F_2$ satisfy the following even-odd conditions:
\begin{align}\label{evenodd}
F_1(u,-v)=F_1(u,v),\ F_2(u,-v)=-F_2(u,v).
\end{align}
\par
Let $J\in\Su$, there is a natural algebra isomorphism from the complex plane to $\C_J$ given by
\begin{align*}
&\Phi_J: \C\longrightarrow \C_J,\\
&u+iv\mapsto u+Jv.
\end{align*}
\par
In order to have a nice theory for slice monogenic functions, one has to put some restrictions on the open sets $\Omega\subset\R^{m+1}$.
\begin{definition}
 Given an open set $D\subset\C$, denote
\begin{align*}
\Omega_D=\bigcup_{J\in\Su}\Phi_J(D)\subset\R^{m+1}.
\end{align*}
If the open set $\Omega\subset\R^{m+1}$ satisfies the form $\Omega=\Omega_D$ for some $D\subset\C$, then we say it is \emph{axially symmetric}.
\end{definition}

Now, we can construct a function defined on $\R^{m+1}$ from a stem function  defined on the complex plane as follows.

\begin{definition}
 Let $F(z)=F_1(z)+iF_2(z)$ be a stem function defined on $D$ with $F_1,F_2:\ D\longrightarrow \Clm$ and $z=u+iv$. Then, $F(z)$ induces a \emph{(left) slice function} $f=\mathcal{I}(F):\Omega_D\longrightarrow \Clm$, defined by
\begin{align*}
f(\bx)=F_1(z)+JF_2(z),\quad \bx=u+Jv=\Phi_J(z)\in\Omega_D\cap \C_J.
\end{align*}
\end{definition}
\begin{remark}
On one hand, the definition above shows that a stem function induces a slice function. On the other hand, a slice function is induced by a unique stem function. Indeed, if a slice function $f(\bx)$ is induced by two stem functions $F(z)=F_1(z)+iF_2(z)$ and $H(z)=H_1(z)+iH_2(z)$. Then, we have that $(F_1-H_1)(z)+J(F_2-H_2)(z)=0$. Notice that $F_1-H_1$ and $F_2-H_2$ satisfy the even-odd conditions \eqref{evenodd}, in other words,
\begin{align*}
(F_1-H_1)(z)=(F_1-H_1)(\overline{z}),\ -(F_2-H_2)(z)=(F_1-H_1)(\overline{z}),
\end{align*}
which immediately gives that $0=(F_1-H_1)(\overline{z})+J(F_2-H_2)(\overline{z})=(F_1-H_1)(z)-J(F_2-H_2)(z)$. Hence,   $F_1=H_1,\ F_2=H_2$ in $\Omega_D$, which certifies the fact that a slice function is induced by a unique stem function.
\end{remark}
Denote the set of (left) slice functions on $\Omega_D$ by
\begin{align*}
\mathcal{S}(\Omega_D):=\{f:\ \Omega_D\rightarrow \Clm\ |\ f=\mathcal{I}(F),\ F:\ D\rightarrow\Clm\otimes\C\ \text{stem function}\}.
\end{align*}

An important property of the slice function is the following representation formula.
 \begin{theorem}[Representation formula]\cite{15}\label{Rep}
Let $D\subset\C$ be a domain, which is invariant under complex conjugation,  and let $\Omega_D\subset\Rmo$ be an axially symmetric domain. Further, let $f:\ \Omega_D\longrightarrow \Clm$ be a slice function. Then, for any $I\in\Su$ and $\bx=u+I_{\bx}v\in\Omega_D$, where $I_{\bx}\in\Su$, we have
\begin{align*}
f(\bx)=\frac{1-I_{\bx}I}{2}f(u+Iv)+\frac{1+I_{\bx}I}{2}f(u-Iv).
\end{align*}
\end{theorem}

Denote the real vector space of slice functions with stem function of class $C^1$ by
\begin{align*}
\mathcal{S}^1(\Omega_D):=\{f=\mathcal{I}(F)\in\mathcal{S}(\Omega_D)\ |\ F\in C^1(D)\}.
\end{align*}
\begin{definition}
Let $D\subset\C$ be a domain, which is invariant under complex conjugation, and $\Omega_D\subset\Rmo$ be an axially symmetric domain. Then, a slice function $f:\ \Omega_D\longrightarrow \Clm$ is called \emph{slice regular (or slice monogenic)} if its stem function $F=F_1+iF_2:\ D\longrightarrow \Clm\otimes\C$ is holomorphic, i.e., its components $F_1,F_2$ satisfy the Cauchy-Riemann equations:
\begin{align*}
\frac{\partial F_1}{\partial u}(z)=\frac{\partial F_2}{\partial v}(z),\ \frac{\partial F_1}{\partial v}(z)=-\frac{\partial F_2}{\partial u}(z),\quad z=u+iv\in D.
\end{align*}
\end{definition}
In \cite{4}, the authors introduced a non-constant coefficients differential operator, which has close connections to slice regular functions. Indeed, the space of slice monogenic (or regular) functions coincides with the kernel space of the  slice Cauchy-Riemann operator $G$ when the domain does not intersect the real line, where $G$ is defined by
\begin{align*}
G=\frac{\partial}{\partial x_0}+\frac{\ubx}{|\ubx|^2}\sum_{j=1}^{m}x_j\frac{\partial}{\partial x_j}.
\end{align*}
This differential operator is slightly different from the operator given in \cite{4} by a factor $|\ubx|^2$, which needs us to be more careful when the domain intersect the real line, which creates singularities for the term $|\ubx|^2$ in the differential operator. Hence, we introduce the notation $\R^{m+1}_*:=\R^{m+1}\backslash{\R}$ for the rest of this article. One might also notice that the operator $G$ also coincides with the operator $2\overline{\vartheta}$ given in \cite{GP}.
\par
Now, we recall the Cauchy kernel for slice monogenic functions as
\begin{align*}
S^{-1}(\bq,\bx)=-(\bq^2-2\Real[\bx]\bq+|\bx|^2)^{-1}(\bq-\overline{\bx}),
\end{align*}
where $\bq^2-2\Real[\bx]\bq+|\bx|^2\neq 0$. This function is left slice monogenic in the variable $\bq$ and right slice monogenic in the variable $\bx$ in its domain of definition, respectively. More details can be found in \cite{7}. The Cauchy kernel for the global slice Cauchy-Riemann operator $G$ was given in \cite[Corollary 2.8]{18} in the following form
\begin{align*}
K(\bq,\bx)=\frac{2S^{-1}(\bq,\bx)}{\omega_{m-1}|\ubx|^{m-1}},
\end{align*}
where $\omega_{m-1}$ is the area of the $(m-1)$-sphere $\Su$. Now, we introduce a Borel-Pompeiu integral formula for the slice Cauchy-Riemann operator given in \cite{18} as follows. The theorem is stated with $G$, since the conjugate slice derivative operator $\frac{\partial}{\partial \bx^c}$ introduced in \cite[Definition 7]{15} equals to $\frac{G}{2}$, when acting on slice functions defined on $\Omega_D\subset\R^{m+1}_*$.
\begin{theorem}[Borel-Pompeiu formula] \cite[Corollary 2.8]{18} \label{BPF} Let $\Omega_D\subset\R^{m+1}_*$ be a bounded axially symmetric domain with smooth boundary $\partial\Omega_D$. If $f:\Omega\longrightarrow \Clm$ is a function in $\mathcal{S}^1(\overline{\Omega_D})$, where $\overline{\Omega_D}$ denotes the Euclidean closure of $\Omega_D$ in $\R^{m+1}$. Then, for any $\bq\in\Omega_D$, we have
\begin{align*}
\int_{\partial\Omega_D}K(\bq,\bx)n({\bx}) f(\bx)d\sigma(\bx)-\int_{\Omega_D}K(\bq,\bx)(Gf)(\bx)dV(\bx)
=2\pi f(\bq),
\end{align*}
where $n(\bx)$ is the outward unit normal vector to the boundary $\partial\Omega_D$, $d\sigma$ is the area element on $\partial\Omega_D$ and $dV$ is the volume element in $\Omega_D$.
In particular, if $Gf(\bx)=0$ for $\bx\in\Omega_D$, we have
\begin{align*}
2\pi f(\bq)=\int_{\partial\Omega_D}K(\bq,\bx)n({\bx}) f(\bx)d\sigma(\bx).
\end{align*}
\end{theorem}
Once we notice that the kernel $K(\bq,\bx)$ is arbitrarily often continuously differentiable with respect to $q_i$ (up to a set with measure zero), we can immediately obtain a Cauchy integral formula for derivatives as follows.
\begin{theorem}[Cauchy integral formula for derivatives]\label{CID}
 Let $\Omega_D\subset\R^{m+1}_*$ be a bounded axially symmetric domain with smooth boundary $\partial\Omega_D$. If $f:\Omega_D\longrightarrow \Clm$ is a function in $C^1(\overline{\Omega_D})\cap \mathcal{S}(\Omega_D)$ and $Gf(\bx)=0$ for $\bx\in\Omega_D$. Then, we have
 \begin{align*}
	2\pi\nabla_{\bq}^{\bo{l}}f(\bq)=\int_{\partial\Omega_D}K_{\bo{l}}(\bq,\bx)n({\bx}) f(\bx)d\sigma(\bx),
\end{align*}
where $\nabla_{\bq}=(\partial_{q_0},\ldots,\partial_{q_m})$ is the gradient, $\bo{l}=(l_0,\ldots,l_m)\in \mathbb{N}^{m+1}$ is a multi-index, $\nabla_{\bq}^{\bo{l}}:=(\partial_{q_0}^{l_0},\ldots,\partial_{q_m}^{l_m})$ and $K_{\bo{l}}(\bq,\bx)=\nabla_{\bq}^{\bo{l}}K(\bq,\bx)$.
\end{theorem}
\begin{remark}
Here, we want to point out that $C^1({\Omega_D})\cap \mathcal{S}(\Omega_D)\neq \mathcal{S}^1({\Omega_D})$. Indeed, in \cite{15}, the authors proved that
$\mathcal{S}^1({\Omega_D})\subset C({\Omega_D})$, but $\ \mathcal{S}^1({\Omega_D})\not\subset C^1({\Omega_D})$, see \cite[Example 4.3]{XS} with $(p,q)=(0,n)$.
\end{remark}
Denote
\begin{align*}
&T_{\Omega_D}f(\bq)=-\frac{1}{2\pi}\int_{\Omega_D}K(\bq,\bx)f(\bx)dV(\bx),\\
&F_{\partial\Omega_D}f(\bq)=\frac{1}{2\pi}\int_{\partial\Omega_D}K(\bq,\bx)n({\bx}) f(\bx)d\sigma(\bx).
\end{align*}
Here, $T_{\Omega_D}$ is usually called the \emph{Teodorescu transform}.
Then the equation in Theorem \ref{BPF}  can be written as
\begin{align*}
	F_{\partial\Omega_D}f(\bq)+T_{\Omega_D}(Gf)(\bq)=f(\bq),\quad \bq\in\Omega_D.
\end{align*}
  In particular, for the  function $f$ with compact support in $\Omega_D$, we have that
\begin{align*}
	T_{\Omega_D}(Gf)(\bq)=f(\bq), \quad \bq\in\Omega_D,
\end{align*}
which suggests that $T_{\Omega_D}$ is a left inverse of $G$ when acts on certain function spaces. This gives rise to a natural question that whether $T_{\Omega_D}$ is also a right inverse of $G$. Thus, the question to study $GT_{\Omega_D}$ arises.

\section{Properties of the Teodorescu transform}
 In this section, we investigate the boundedness of the operator $T_{\Omega_D}$. Let $L^p(\Omega_D)$ be the classical Lebesgue $L^p$ space on $\Omega_D$.
From now on,   the symbol $C$  represents some positive  constant, which is independent of the function $f$ and might varies in different places, and we will not specify the parameters $m,p,\epsilon_o,\epsilon_1,a,b,\bq$ that the constant $C$ depends on unless it depends on the domain $\Omega_D$.
\begin{proposition}\label{BoundT}
Let $\Omega= {\Omega_D}\subset\R^{m+1}_*$ be a bounded axially symmetric domain and $f\in L^p(\Omega_D)$ with   $p>\max\{m,2\}$. Then we have
\begin{enumerate}
\item the integral $T_{\Omega_D}f(\bq)$ exists everywhere in $\R^{m+1}$;
\item $GT_{\Omega_D}f=0$ in $\R^{m+1}_*\backslash
\overline{\Omega_D}$;
\item
$
||T_{\Omega_D}f||_{L^p}\leq C(\Omega_D)||f||_{L^p}.
$
\end{enumerate}
\end{proposition}
\begin{proof}
 Firstly, we notice that by Theorem \ref{BPF}
\begin{align*}
{2\pi}|T_{\Omega_D}f(\bq)|=&\bv\int_{\Omega_D}K(\bq,\bx)f(\bx)dV(\bx)\bv\\
\leq &\bigg[\int_{\Omega_D}|K(\bq,\bx)|^{p'}dV(\bx)\bigg]^{\frac{1}{p'}}||f||_{L^p},
\end{align*}
where $\frac{1}{p}+\frac{1}{p'}=1,\ p>1$. Now, we consider from Theorem \ref{Rep}
\begin{align*}
&\int_{\Omega_D}|K(\bq,\bx)|^{p'}dV(\bx)\\
=&C\int_{\mathbb{S}^+}\int_{\Omega_I}\bv\frac{S^{-1}(\bq,\bx)}{|\ubx|^{m-1}}\bv^{p'}|\ubx|^{m-1}dV_I(\bx)dS(I)\nonumber\\
=&C \int_{\mathbb{S}^+}\int_{\Omega_I}\bv \alpha\frac{1}{\bx-\bq_I}+\beta\frac{1}{\bx-\bq_{-I}}\bv^{p'}|\ubx|^{-(m-1)(p'-1)}dV_I(\bx)dS(I)\nonumber\\
\leq& C\int_{\mathbb{S}^+}\int_{\Omega_I}(|\bx-\bq_I|^{-p'}+|\bx-\bq_{-I}|^{-p'})|\ubx|^{-(m-1)(p'-1)}dV_I(\bx)dS(I),
\end{align*}
where $\bq=q_0+I_{\bq}|\underline{\bq}|$, $\bq_{\pm I}=q_0\pm I|\underline{\bq}|$, $\alpha=\frac{1-I_{\bq}I}{2},\ \beta=\frac{1+I_{\bq}I}{2}$, $dV_I (x)$ stands for the area element on the complex plane $\mathbb{C}_I$ and $dS(I)$ is the area element on the half unit sphere $\mathbb{S}^+:=\{\bx\in\mathbb{S}: x_m>0\}$. Next, we only need to verify that the integral
\begin{align}\label{BTE}
\Psi=\int_{\mathbb{S}^+}\int_{\Omega_I}|\bx-\bq_I|^{-p'}|\ubx|^{-(m-1)(p'-1)}dV_I(\bx)dS(I)
\end{align}
is finite for $p>\max\{m,2\}$, the argument for the other one is similar.

It is worth pointing out that the inner integral in \eqref{BTE} is independent of $I$, which implies that
\begin{align*}
 \Psi=\Psi(q_{0},|\underline{\bq}|)=C\int_{\Omega_I}|\bx-\bq_I|^{-p'}|\ubx|^{-(m-1)(p'-1)}dV_I(\bx).
\end{align*}
For $I\in \mathbb{S}$, denote
\begin{align*}
E_I=\{u+Iv: r<u<R,\ -M<v<M\}.
\end{align*}
Notice that $\Omega=\Omega_D\subset\R^{m+1}_*$ is bounded, then there exist $r,R,M \in\R$ such that $ \Omega_I\subset E_{I}$. Hence, we have
\begin{align}\label{EI}
\Psi \leq C \int_{E_I}|\bx-\bq_I|^{-p'}|\ubx|^{-(m-1)(p'-1)}dV_I(\bx).
\end{align}
   The singularities of the integral \eqref{EI} occur in the following two cases.
\begin{enumerate}
\item The singular point $\bx=\bq_I$,
\item Points $\bx$ on the real line, which have $\ubx=0$.
\end{enumerate}
Let $B(\bq_I,\epsilon_0)=\{ z=u+Iv: |z-\bq_I|< \epsilon_0\}\subset E_I$ be a neighborhood of $\bq_I$ in the plane $\mathbb{C}_I$ with a sufficiently small $\epsilon_0$ and $E_{I}(\epsilon_1)=\{u+Iv\in E_I:\ r<u<R,\ -\epsilon_1<v<\epsilon_1\}$ be a strip neighborhood of the interval $(r,R)$ in $E_I$. The finiteness of the integral \eqref{EI} is equivalent to
\begin{align*}
\int_{B(\bq_I,\epsilon_0)\cup E_{I}(\epsilon_1) }|\bx-\bq_I|^{-p'}|\ubx|^{-(m-1)(p'-1)}dV_I(\bx)<+\infty,
\end{align*}
which is easy to obtain with applying spherical coordinates to the integral over $B(\bq_I,\epsilon_0)$  and Cartesian coordinates to the integral over $E_{I}(\epsilon_1)$. Indeed, on the one hand, we let $\bx=\bq_I+re^{I\theta}$, then we have, for $p>2$,
\begin{align*}
&\int_{B(\bq_I,\epsilon_0)}|\bx-\bq_I|^{-p'}|\ubx|^{-(m-1)(p'-1)}dV_I(\bx)\\
\leq& C(\Omega_D)\int_0^{\epsilon_0}\int_{0}^{2\pi}r^{1-p'}drd\theta<+\infty.
\end{align*}
 On the other hand,   denote $\eta=|\ubx|$, then we have for $p>m$
\begin{align*}
&\int_{E_{I}(\epsilon_1)}|\bx-\bq_I|^{-p'}|\ubx|^{-(m-1)(p'-1)}dV_I(\bx)\\
\leq&C(\Omega_D)\int_{r}^R\int_{-\epsilon_1}^{\epsilon_1}\eta^{-(m-1)(p'-1)}dx_0d\eta<+\infty.
\end{align*}
\par
Therefore, for fixed $\bq\in\Omega_D$, we can easily have
\begin{align*}
|T_{\Omega_D}f(\bq)|\leq C(\Omega_D)||f||_{L^p},
\end{align*}
which immediately leads to the fact that $T_{\Omega_D}f(\bq)$ exists everywhere in $\R^{m+1}$.
\par
Moreover, since $T_{\Omega_D}f$ has no singular points in $\R^{m+1}_*\backslash
\overline{\Omega_D}$ and it is obvious that $GS^{-1}(\cdot,\bx)=0$, we immediately have that $GT_{\Omega_D}f=0$ in $\R^{m+1}_*\backslash
\overline{\Omega_D}$.
\par
With the previous argument for $|T_{\Omega_D}f(\bq)|$, we can see that estimating the $L^p$ norm of $T_{\Omega_D}f$ is equivalent to show the finiteness of the $L^p$ norm of the integral \eqref{BTE}. However, to prove the third statement, we can not use the estimate for $|T_{\Omega_D}f(\bq)|$ obtained above, since the constant $C$ depends on $\Omega_D$, which implies that we can not deal with the singularities as above for all $\bq\in\Omega_D$ uniformly when $\partial\Omega_D\cap\mathbb{R}\neq\emptyset$. Hence, we need a more subtle argument to estimate the $L^p$ norm of the integral \eqref{BTE}.
\par
Firstly, we choose a sufficiently large $r_1>0$, such that $\Omega_D\subset B^{m+1}(\bq,r_1)$ for any $\bq\in\Omega_D$, where $B^{m+1}(\bq,r_1)$ stands for the ball centered at $\bq$ with radius $r
_1$ in $\mathbb{R}^{m+1}$. Let $\bx=\bq_I+re^{I\theta}$, $v=|\underline{\bq}|$, and $B(\bq_I,r_1)=B^{m+1}(\bq,r_1)\cap\mathbb{C}_I$, then, we can see that $|\underline{\bx}|=|v+r\sin\theta|$. Hence,
\begin{align*}
\Psi\leq& C\int_{B(\bq_I,r_1)}|\bx-\bq_I|^{-p'}|\ubx|^{-(m-1)(p'-1)}dV_I(\bx)\\
=&C\int_0^{r_1}\int_{0}^{2\pi}r^{1-p'}|v+r\sin\theta|^{-(m-1)(p'-1)}drd\theta.
\end{align*}
\par
Next, let $a,b>0$ such that $\Omega_J\subset\{\bq=u+Jv\in\R^{m+1}:-a<u<a,0<v<b\}$ for all $J\in\mathbb{S}$. Now, we check the finiteness of the $L^p$ norm of the  integral above as follows.
\begin{align}\label{I12T}
&\int_{\Omega_D}|\Psi|^pdV(\bq)=C\int_{\Omega_I}|\Psi(q_{0},|\underline{\bq}|)|^p |\underline{\bq}|^{m-1}dV_{I}(\bq)\nonumber\\
\leq&C\int_{-a}^a\int_{0}^b\bigg\vert\int_0^{r_1}\int_{0}^{2\pi}r^{1-p'}|v+r\sin\theta|^{(m-1)(1-p')}drd\theta \bigg\vert^p v^{m-1}dudv\nonumber\\
\leq& C\int_{0}^b\bigg\vert\int_{0}^{2\pi}\int_0^{r_1}r^{1-p'}|v+r\sin\theta|^{(m-1)(1-p')}drd\theta \bigg\vert^p v^{m-1}dv\nonumber\\
\leq&C\int_{0}^b\bigg\vert\int_{0}^{\pi}\int_0^{r_1}r^{1-p'}|v+r\sin\theta|^{(m-1)(1-p')}drd\theta \bigg\vert^p v^{m-1}dv\nonumber\\
&+C\int_{0}^b\bigg\vert\int_{\pi}^{2\pi}\int_0^{r_1}r^{1-p'}|v+r\sin\theta|^{(m-1)(1-p')}drd\theta \bigg\vert^p v^{m-1}dv=:I_1+I_2.
\end{align}
Now, we estimate the integrals of $I_1$ and $I_2$, respectively.
\par
To estimate $I_1$, we notice that $v+r\sin\theta\geq 0$ and $(m-1)(1-p')<0$, which give us that $|v+r\sin\theta|^{(m-1)(1-p')}\leq v^{(m-1)(1-p')}$. Hence, we have
\begin{align*}
I_1\leq& C\int_{0}^b\bigg\vert\int_0^{r_1}r^{1-p'}v^{(m-1)(1-p')}dr \bigg\vert^p v^{m-1}dv\\
=&C\bigg(\int_0^{r_1}r^{1-p'}dr\bigg)^p\int_0^bv^{(m-1)(1-p')p+m-1}dv,
\end{align*}
and this integral converges only if $1-p'>-1$ and $(m-1)(1-p')p+m-1>-1$, which give    $p>\max\{m,2\}$.
\par
To estimate $I_2$, let $\rho=-\sin\theta>0$, without loss of generality, we assume that $0<v<\rho r_1$, otherwise, it is a special case of the following argument. Then, we have
\begin{align}\label{I12}
&\int_0^{r_1}r^{1-p'}|v-\rho r|^{(m-1)(1-p')}dr\nonumber\\
 =&\int_0^{\frac{v}{\rho}}r^{1-p'}(v-\rho r)^{(m-1)(1-p')}dr +\int_{\frac{v}{\rho}}^{r_1}r^{1-p'}(\rho r-v)^{(m-1)(1-p')}dr \nonumber\\
 =&:\Psi_1+\Psi_2.
\end{align}
To calculate $\Psi_1$, let $r=\frac{v}{\rho}\tau$, we have
\begin{align*}
\Psi_1=&\int_{0}^1\bigg(\frac{v}{\rho}\tau\bigg)^{1-p'}v^{(m-1)(1-p')}(1-\tau)^{(m-1)(1-p')}\frac{v}{\rho}d\tau\\
=&\rho^{p'-2}v^{2-p'+(m-1)(1-p')}\int_0^1\tau^{1-p'}(1-\tau)^{(m-1)(1-p')}d\tau\\
=&c_{m,p}\rho^{p'-2}v^{2-p'+(m-1)(1-p')},
\end{align*}
where $c_{m,p}=B(2-p',(m-1)(1-p')+1)$ is a Beta function and it converges if and only if $2-p'>0$ and $(m-1)(1-p')+1>0$, that is $p>\max\{m,2\}$.
\par
To   estimate $\Psi_2$, let $\mu=\rho r-v$, then we have
\begin{align*}
 \Psi_2 \leq & \bigg(\frac{v}{\rho}\bigg)^{1-p'}\int^{r_1}_{\frac{v}{\rho}}(\rho r-v)^{(m-1)(1-p')}dr\\
=&\rho^{p'-2}v^{1-p'}\int_0^{\rho r_1-v}\mu^{(m-1)(1-p')}d\mu\\
=&c'_{m,p}\rho^{p'-2}v^{1-p'}\mu^{(m-1)(1-p')+1}\bigg\vert_0^{\rho r_1-v}\\
=&c'_{m,p}\rho^{p'-2}v^{1-p'}(\rho r_1-v)^{(m-1)(1-p')+1},
\end{align*}
where $c'_{m,p}=\frac{1}{(m-1)(1-p')+1}$, and the last equation relies on the condition $(m-1)(1-p')>-1$, that is $p>m$. Plugging $\Psi_1$ and $\Psi_2$ into \eqref{I12} and \eqref{I12T}, we  obtain
$$I_2\leq C\int_0^b\bigg\vert\int_{\pi}^{2\pi}c_{m,p}\rho^{p'-2} v^{2-p'+(m-1)(1-p')}+c_{m,p}'v^{1-p'}|\rho r_1-v|^{(m-1)(1-p')+1}d\theta\bigg\vert^pv^{m-1} dv.
$$
Recall that when we estimate $\Psi_1$ and $\Psi_2$, we used the conditions
$$(m-1)(1-p')+1>0,\ p>\max\{m,2\},$$  which imply  for $0<v<b$
$$|\rho r_1-v|^{(m-1)(1-p')+1}\leq (r_1+b)^{(m-1)(1-p')+1}=:C_1(m,p),$$
Hence,
\begin{align*}
I_2\leq& C\int_0^b\bigg\vert\int_{\pi}^{2\pi}\rho^{p'-2} v^{2-p'+(m-1)(1-p')}+C_1(m,p)v^{1-p'}d\theta\bigg\vert^pv^{m-1} dv\\
=&C \int_0^b\bigg\vert\int_{\pi}^{2\pi}\rho^{p'-2}d\theta\cdot v^{2-p'+(m-1)(1-p')}+C_1'(m,p)v^{1-p'}\bigg\vert^pv^{m-1} dv\\
\leq& C\bigg(\int_0^b\bigg\vert \int_{\pi}^{2\pi}\rho^{p'-2}d\theta\bigg\vert^p v^{[2-p'+(m-1)(1-p')]p+m-1}dv+\int_0^bv^{(1-p')p+m-1}dv\bigg).
\end{align*}
The convergence of the integrals above requires that $$ [2-p'+(m-1)(1-p')]p+m-1>-1,\ (1-p')p+m-1>-1,$$
or 
 $$p>\frac{2m+1}{m+1},\ p>\frac{m}{m-1},$$
and the convergence of
\begin{align*}
\int_{\pi}^{2\pi}\rho^{p'-2}d\theta=\int_{\pi}^{2\pi}(-1)^{p'-2}\sin^{p'-2}\theta d\theta.
\end{align*}
These requirements hold  naturally   under  the condition  $p>\max\{m,2\}$, and we eventually prove the third statement.
\end{proof}
\begin{remark}
When $m=1$, in other words, it is the case of the complex plane, our result reduces to $p>2$, which coincides with the classical result in the complex plane, see \cite[Proposition 8.1]{23}.
\end{remark}
For convenience,  we denote the slice Teodorescu transform by
\begin{align*}
&T_{\Omega_I}f(\bq)=-\frac{1}{2\pi}\int_{\Omega_I}S^{-1}(\bq,\bx)f(\bx)dV_I(\bx).
\end{align*}
We need the following well-known theorem from measure theory for interchanging differentiation and integration.
\begin{theorem}\label{inter}\cite{12}
Suppose that $f:X\times[a,b]\longrightarrow\mathbb{C},\ (-\infty<a<b<+\infty)$ and $f(\cdot,t):X\longrightarrow\mathbb{C}$ is integrable for each $t\in [a,b]$. Let
\begin{align*}
F(t)=\int_{X}f(x,t)d\mu(x),
\end{align*}
and
\begin{enumerate}
\item Suppose that $\displaystyle\frac{\partial f}{\partial t}$ exists,
\item $\exists g\in L^1(\mu)$ such that $\bigg\vert\displaystyle\frac{\partial f}{\partial t}(x,t)\bigg\vert\leq g(x)$ for all $x$ and $t$.
\end{enumerate}
Then $F$ is differentiable and
\begin{align*}
F'(t)=\int_X\frac{\partial f}{\partial t}(x,t)d\mu(x).
\end{align*}
\end{theorem}
We also need the Gauss theorem for $\partial_{\bx_I}$  on a complex plane as follows.
\begin{theorem}\cite{23}\label{Gauss}
Let $\Omega_I\subset\mathbb{C}_I$ be a domain, and $f(\bx),\ g(\bx)\in C^1(\overline{\Omega_I})$. Then, we have
\begin{align*}
	\int_{\Omega_I}\big(f(\bx)\partial_{\bx_I}\big)g(\bx)+f(\bx)\big(\partial_{\bx_I}g(\bx)\big)dV_I(\bx)=\int_{\partial\Omega_I}f(\bx)\overline{d\bx^*}g(\bx),
\end{align*}
where $\overline{d\bx^*}=Id\bx$ and $d\bx$ is the line element on $\partial\Omega_I$.
\end{theorem}
\begin{remark}
The standard Gauss theorem is stated for the $\overline{\partial}$ operator. The Gauss theorem above can be considered as a conjugate version of the standard one, which can be obtained by taking the conjugate of the standard Gauss theorem.
\end{remark}
Now, we claim that $G$ is a left inverse for $T_{\Omega_I}$ as follows.
\begin{theorem}\label{thmOI}
Let $\Omega_D\subset \R_*^{m+1}$ be a  bounded axially symmetric domain, $f\in C^1(\overline{\Omega_D})$, then, for any $\bq\in\Omega_D$, we have
\begin{align*}
2\pi \partial_{q_0}T_{\Omega_I}f(\bq)&=\int_{\Omega_I}\frac{\partial S^{-1}(\bq,\bx)}{\partial q_0}f(\bx)dV_I(\bx)+\pi[\alpha f(\bq_I)+\beta f(\bq_{-I})],\\
2\pi \partial_{q_i}T_{\Omega_I}f(\bq)
&=\frac{q_i}{|\ubq|}\bigg[\int_{\Omega_I}\frac{\partial S^{-1}(\bq,\bx)}{\partial{\zeta}}f(\bx)dV_I(\bx)-\pi [\alpha I f(\bq_I)-\beta {I}f(\bq_{-I})]\bigg],\\
G_{\bq}T_{\Omega_I}f(\bq)&=\alpha f(\bq_I)+\beta f(\bq_{-I}),
\end{align*}
where $\bq=q_0+I_{\bq}|\underline{\bq}|$, $\zeta=|\underline{\bq}|$, $\bq_{\pm I}=q_0\pm I\zeta$, $G_{\bq}$ is the $G$ operator given with respect to the variable $\bq$, and
\begin{align*}
\alpha=\frac{1-I_{\bq}I}{2},\ \beta=\frac{1+I_{\bq}I}{2}.
\end{align*}
Further, if $f\in\mathcal{S}(\Omega_D)$ as well, we have
\begin{align*}
G_{\bq}T_{\Omega_I}f(\bq)=f(\bq).
\end{align*}
\end{theorem}
The method we applied below is to rewrite $\bq$ in terms of $\bq_I$ and $\bq_{-I}$ by the representation formula, which gives rise to two singular integral operators, and it makes sense only as Cauchy's principal values.
\begin{proof}
Firstly, we denote $\bx=x_0+I|\ubx|=:x_0+I\eta$, $\overline{\partial}_{\bq_I}=\frac{1}{2}(\partial_{q_0}+I\partial_{\zeta})$ and ${\partial}_{\bq_I}=\frac{1}{2}(\partial_{q_0}-I\partial_{\zeta})$. We notice that
\begin{align*}
S^{-1}(\bq,\bx)
=\alpha (\bx-\bq_I)^{-1}+\beta (\bx-\bq_{-I})^{-1},
\end{align*}
where $(\bx-\bq_I)^{-1}$ is the Cauchy kernel on the plane $\mathbb{C}_I$, and
\begin{align*}
(\bx-\bq_I)^{-1}=\frac{\overline{\bx-\bq_I}}{|\bx-\bq_I|^2}=-\partial_{\bq_I}\ln|\bx-\bq_I|=\partial_{\bx_I}\ln|\bx-\bq_I|.
\end{align*}
Since $T_{\Omega_I}f$ is a singular integral, which only makes sense as a Cauchy principal value, let $B_{\epsilon}=B(\bq_I,\epsilon)\cup B(\bq_{-I},\epsilon)\subset\Omega_I$ for a sufficiently small $\epsilon>0$. Then, we have
\begin{align*}
&-2\pi T_{\Omega_I}f(\bq)=\int_{\Omega_I}S^{-1}(\bq,\bx)f(\bx)dV_I(\bx)\\
=&\lim_{\epsilon\rightarrow 0}\int_{\Omega_I\backslash B_{\epsilon}}S^{-1}(\bq,\bx)f(\bx)dV_I(\bx)\\
=&\lim_{\epsilon\rightarrow 0}\int_{\Omega_I\backslash B_{\epsilon}}\big(\alpha (\bx-\bq_I)^{-1}+\beta (\bx-\bq_{-I})^{-1}\big)f(\bx)dV_I(\bx)\\
=&\lim_{\epsilon\rightarrow 0}\int_{\Omega_I\backslash B_{\epsilon}}\bigg[\big(\alpha \ln|\bx-\bq_I|+\beta \ln|\bx-\bq_{-I}|\big)\partial_{\bx_I}\bigg]f(\bx)dV_I(\bx)\\
=&\lim_{\epsilon\rightarrow 0}\bigg[-\int_{\Omega_I\backslash B_{\epsilon}}\big(\alpha \ln|\bx-\bq_{I}|+\beta \ln|\bx-\bq_{-I}|\big)\big(\partial_{\bx_I}f(\bx)\big)dV_I(\bx)\\
&+\bigg(\int_{\partial\Omega_I}-\int_{\partial B_{\epsilon}}\bigg)\big(\alpha \ln|\bx-\bq_{I}|+\beta \ln|\bx-\bq_{-I}|\big)\overline{d\bx^*}f(\bx)\bigg]\\
=&-\int_{\Omega_I}\big(\alpha \ln|\bx-\bq_{I}|+\beta \ln|\bx-\bq_{-I}|\big)\big(\partial_{\bx_I}f(\bx)\big)dV_I(\bx)\\
&+\int_{\partial\Omega_I}\big(\alpha \ln|\bx-\bq_{I}|+\beta \ln|\bx-\bq_{-I}|\big)\overline{d\bx^*}f(\bx).
\end{align*}
It is easy to prove that we can interchange differentiation and the integration above with Theorem \ref{inter}. Indeed, since $f\in C^1(\overline{\Omega_D})$, which implies that $f$ is bounded. Further, the homogeneity of
\begin{align*}
\partial_{q_0}\big(\alpha \ln|\bx-\bq_{I}|+\beta \ln|\bx-\bq_{-I}|\big)=\alpha\frac{q_0-x_0}{|\bq_I-\bx|^2}+\beta\frac{q_0-x_0}{|\bq_{-I}-\bx|^2}
\end{align*}
suggests that it is integrable with respect to $\bx$, which means that the two conditions of Theorem \ref{inter} are satisfied. Hence, we have
\begin{align}\label{TOI}
2\pi \partial_{q_0}T_{\Omega_I}f(\bq)=&\int_{\Omega_I}\bigg(\alpha\frac{q_0-x_0}{|\bq_I-\bx|^2}+\beta\frac{q_0-x_0}{|\bq_{-I}-\bx|^2}\bigg)\big(\partial_{\bx_I}f(\bx)\big)dV_I(\bx)\nonumber\\
&-\int_{\partial\Omega_I}\bigg(\alpha\frac{q_0-x_0}{|\bq_I-\bx|^2}+\beta\frac{q_0-x_0}{|\bq_{-I}-\bx|^2}\bigg)\overline{d\bx^*}f(\bx).
\end{align}
Further, with the help of Theorem \ref{Gauss}, we know that
\begin{align*}
&\bigg(\int_{\partial\Omega_I}-\int_{\partial B_{\epsilon}}\bigg)\bigg(\alpha\frac{q_0-x_0}{|\bq_I-\bx|^2}+\beta\frac{q_0-x_0}{|\bq_{-I}-\bx|^2}\bigg)\overline{d\bx^*}f(\bx)\\
=&\int_{\Omega_I\backslash B_{\epsilon}}\bigg[\bigg(\alpha\frac{q_0-x_0}{|\bq_I-\bx|^2}+\beta\frac{q_0-x_0}{|\bq_{-I}-\bx|^2}\bigg)\partial_{\bx_I}\bigg]f(\bx)\\
&\quad\quad\quad+\bigg(\alpha\frac{q_0-x_0}{|\bq_I-\bx|^2}+\beta\frac{q_0-x_0}{|\bq_{-I}-\bx|^2}\bigg)\big(\partial_{\bx_I}f(\bx)\big)dV_I(\bx).
\end{align*}
Substituting the equation above into \eqref{TOI}, we have
\begin{align}\label{Homo}
2\pi \partial_{q_0}T_{\Omega_I}f(\bq)=&\lim_{\epsilon\rightarrow 0}\int_{B_{\epsilon}}\bigg(\alpha\frac{q_0-x_0}{|\bq_I-\bx|^2}+\beta\frac{q_0-x_0}{|\bq_{-I}-\bx|^2}\bigg)\big(\partial_{\bx_I}f(\bx)\big)dV_I(\bx)\nonumber\\
&-\int_{\Omega_I\backslash B_{\epsilon}}\bigg[\bigg(\alpha\frac{q_0-x_0}{|\bq_I-\bx|^2}+\beta\frac{q_0-x_0}{|\bq_{-I}-\bx|^2}\bigg)\partial_{\bx_I}\bigg]f(\bx)dV_I(\bx)
\nonumber\\&-\int_{\partial B_{\epsilon}}\bigg(\alpha\frac{q_0-x_0}{|\bq_I-\bx|^2}+\beta\frac{q_0-x_0}{|\bq_{-I}-\bx|^2}\bigg)\overline{d\bx^*}f(\bx).
\end{align}
From the homogeneity of $\displaystyle\frac{q_0-x_0}{|\bq_I-\bx|^2}$ and $\displaystyle\frac{q_0-x_0}{|\bq_{-I}-\bx|^2}$, one can easily show that
\begin{align*}
\lim_{\epsilon\rightarrow 0}\int_{B_{\epsilon}}\bigg(\alpha\frac{q_0-x_0}{|\bq_I-\bx|^2}+\beta\frac{q_0-x_0}{|\bq_{-I}-\bx|^2}\bigg)\big(\partial_{\bx_I}f(\bx)\big)dV_I(\bx)=0.
\end{align*}
Indeed, we can see that
\begin{align*}
\lim_{\epsilon\rightarrow 0}\int_{B_{\epsilon}}\frac{q_0-x_0}{|\bq_I-\bx|^2}\big(\partial_{\bx_I}f(\bx)\big)dV_I(\bx)=\lim_{\epsilon\rightarrow 0}\int_{B(\bq_I,{\epsilon})}\frac{q_0-x_0}{|\bq_I-\bx|^2}\big(\partial_{\bx_I}f(\bx)\big)dV_I(\bx).
\end{align*}
By changing to the spherical coordinates $\bx=\bq_I+r e^{I\theta}$ with $dV_I(\bx)=r d rd\theta$ and $f\in C^1(\overline{\Omega_D})$, which implies that $|\partial_{\bx_I}f(\bx)|$ is bounded, we can immediately see that
$$\lim_{\epsilon\rightarrow 0}\int_{B_{\epsilon}}\frac{q_0-x_0}{|\bq_I-\bx|^2}\big(\partial_{\bx_I}f(\bx)\big)dV_I(\bx)=0.$$
Further, with a similar argument as in \cite[Theorem 8.2]{23}, let $\bx=\bq_I+\epsilon\bo{t}$, with $|\bo{t}|=1$ and $\bo{t}\in\C_I$. We can rewrite
\begin{align*}
&\lim_{\epsilon\rightarrow 0}\int_{\partial B_{\epsilon}}\frac{q_0-x_0}{|\bq_I-\bx|^2}\overline{d\bx^*} f(\bx)
=\lim_{\epsilon\rightarrow 0}\int_{\partial B(\bq_I,{\epsilon})}\frac{q_0-x_0}{|\bq_I-\bx|^2}\overline{d\bx^*} f(\bx)\\
=&\lim_{\epsilon\rightarrow 0}\int_{\partial B(\bq_I,{\epsilon})}\frac{q_0-x_0}{|\bq_I-\bx|^2}\overline{d\bx^*} f(\bq_I)
+\lim_{\epsilon\rightarrow 0}\int_{\partial B(\bq_I,{\epsilon})}\frac{q_0-x_0}{|\bq_I-\bx|^2}\overline{d\bx^*} (f(\bx)-f(\bq_I)).
\end{align*}
By the continuity of $f$ at $\bq_I$, we can easily see that the second integral above is zero. Now, we consider
\begin{align*}
\lim_{\epsilon\rightarrow 0}\int_{\partial B(\bq_I,{\epsilon})}\frac{q_0-x_0}{|\bq_I-\bx|^2}\overline{d\bx^*} f(\bq_I)=\int_{|\bo{t}|=1}-t_0\overline{d\bo{t}^*}f(\bq_I),
\end{align*}
and using Theorem \ref{Gauss}, the integral above becomes
\begin{align*}
\int_{|\bo{t}|<1}-(\partial_{\bo{t}_I}t_0)dV_I(\bo{t})f(\bq_I)=-\int_{|\bo{t}|<1}dV_I(\bo{t})f(\bq_I)=-\pi f(\bq_I).
\end{align*}
Similarly, we can also obtain
\begin{align*}
\lim_{\epsilon\rightarrow 0}\int_{\partial B_{\epsilon}}\frac{q_0-x_0}{|\bq_{-I}-\bx|^2}\overline{d\bx^*} f(\bx)=-\pi  f(\bq_{-I}).
\end{align*}
 Hence, we have that
\begin{align*}
\lim_{\epsilon\rightarrow 0}\int_{\partial B_{\epsilon}}\bigg(\alpha\frac{q_0-x_0}{|\bq_I-\bx|^2}+\beta\frac{q_0-x_0}{|\bq_{-I}-\bx|^2}\bigg)\overline{d\bx^*} f(\bx)=-\pi [\alpha f(\bq_I)+\beta f(\bq_{-I})].
\end{align*}
Further, we notice that
\begin{align*}
\frac{q_0-x_0}{|\bq_{\pm I}-\bx|^2}\partial_{\bx_I}=\frac{\partial}{\partial q_0}\ln|\bq_{\pm I}-\bx|\partial_{\bx_I}=\frac{\partial}{\partial q_0}\partial_{\bx_I}\ln|\bq_{\pm I}-\bx|=\frac{\partial}{\partial q_0}\frac{1}{\bx-\bq_{\pm I}},
\end{align*}
which gives us that
\begin{align}\label{pq0}
&2\pi \partial_{q_0}T_{\Omega_I}f(\bq)\nonumber\\=&\int_{\Omega_I}\bigg[\frac{\partial}{\partial q_0}\bigg(\alpha\frac{1}{\bx-\bq_{ I}}+\beta\frac{1}{\bx-\bq_{-I}}\bigg)\bigg]f(\bx)dV_I(\bx)+\pi[\alpha f(\bq_I)+\beta f(\bq_{-I})]\nonumber\\
=&\int_{\Omega_I}\frac{\partial}{\partial q_0}S^{-1}(\bq,\bx)f(\bx)dV_I(\bx)+\pi [\alpha f(\bq_I)+\beta f(\bq_{-I})].
\end{align}
Next, we consider the derivatives of $\partial_{q_i}T_{\Omega_I}f(\bq),\ i=1,\ldots,m$. Firstly, we notice that
\begin{align*}
\partial_{q_i}=\frac{\partial\zeta}{\partial q_i}\frac{\partial}{\partial\zeta}=\frac{q_i}{|\ubq|}\frac{\partial}{\partial\zeta}.
\end{align*}
Then, with a similar argument as applied to $\partial_{q_0}T_{\Omega_I}f(\bq)$, we have
\begin{align}\label{PQI}
&2\pi \partial_{q_i}T_{\Omega_I}f(\bq)\nonumber\\
=&\frac{q_i}{|\ubq|}\frac{\partial}{\partial\zeta}\bigg[\int_{\Omega_I}\big(\alpha \ln|\bx-\bq_{I}|+\beta \ln|\bx-\bq_{-I}|\big)\big(\partial_{\bx_I}f(\bx)\big)dV_I(\bx)\nonumber\\
&-\int_{\partial\Omega_I}\big(\alpha \ln|\bx-\bq_{I}|+\beta \ln|\bx-\bq_{-I}|\big)\overline{d\bx^*}f(\bx)\bigg]\nonumber\\
=&\frac{q_i}{|\ubq|}\bigg[\int_{\Omega_I}\bigg(\alpha\frac{\zeta-\eta}{|\bq_I-\bx|^2}+\beta\frac{\zeta+\eta}{|\bq_{-I}-\bx|^2}\bigg)\big(\partial_{\bx_I}f(\bx)\big)dV_I(\bx)\nonumber\\
&-\int_{\partial \Omega_I}\bigg(\alpha\frac{\zeta-\eta}{|\bq_I-\bx|^2}+\beta\frac{\zeta+\eta}{|\bq_{-I}-\bx|^2}\bigg)\overline{d\bx^*}f(\bx)\bigg].
\end{align}
Further, Theorem \ref{Gauss} tells us that
\begin{align*}
&\bigg(\int_{\partial \Omega_I}-\int_{\partial B_{\epsilon}}\bigg)\bigg(\alpha\frac{\zeta-\eta}{|\bq_I-\bx|^2}+\beta\frac{\zeta+\eta}{|\bq_{-I}-\bx|^2}\bigg)\overline{d\bx^*}f(\bx)\\
=&\int_{\Omega_I\backslash B_{\epsilon}}\bigg[\bigg(\alpha\frac{\zeta-\eta}{|\bq_I-\bx|^2}+\beta\frac{\zeta+\eta}{|\bq_{-I}-\bx|^2}\bigg)\partial_{\bx_I}\bigg]f(\bx)\\
&+\bigg(\alpha\frac{\zeta-\eta}{|\bq_I-\bx|^2}+\beta\frac{\zeta+\eta}{|\bq_{-I}-\bx|^2}\bigg)\big(\partial_{\bx_I}f(\bx)\big)dV_I(\bx).
\end{align*}
Substituting the equation above into \eqref{PQI} to obtain
\begin{align*}
2\pi 	\partial_{q_i}T_{\Omega_I}f(\bq)=&\frac{q_i}{|\ubq|}\bigg(-\int_{\partial B_{\epsilon}}\bigg(\alpha\frac{\zeta-\eta}{|\bq_I-\bx|^2}+\beta\frac{\zeta+\eta}{|\bq_{-I}-\bx|^2}\bigg)\overline{d\bx^*}f(\bx)\\
	&+\int_{B_{\epsilon}}\bigg(\alpha\frac{\zeta-\eta}{|\bq_I-\bx|^2}+\beta\frac{\zeta+\eta}{|\bq_{-I}-\bx|^2}\bigg)\big(\partial_{\bx_I}f(\bx)\big)dV_I(\bx)\\
	&-\int_{\Omega_I\backslash B_{\epsilon}}\bigg[\bigg(\alpha\frac{\zeta-\eta}{|\bq_I-\bx|^2}+\beta\frac{\zeta+\eta}{|\bq_{-I}-\bx|^2}\bigg)\partial_{\bx_I}\bigg]f(\bx)dV_I(\bx)\bigg).
\end{align*}
We notice that
\begin{align*}
\frac{\zeta\pm\eta}{|\bq_{\pm I}-\bx|^2}=\ln|\bq_{\pm I}-\bx|\frac{\partial}{\partial{\zeta}},
\end{align*}
which leads to
\begin{align*}
\frac{\zeta\pm\eta}{|\bq_{\pm I}-\bx|^2}\partial_{\bx_I}=\ln|\bq_{\pm I}-\bx|\frac{\partial}{\partial{\zeta}}\partial_{\bx_I}=\ln|\bq_{\pm I}-\bx|\partial_{\bx_I}\frac{\partial}{\partial{\zeta}}=\frac{1}{\bx-\bq_{\pm I}}\frac{\partial}{\partial{\zeta}}.
\end{align*}
Hence, we have that
\begin{align*}
&2\pi \partial_{q_i}T_{\Omega_I}f(\bq)=\lim_{\epsilon\rightarrow 0}\frac{q_i}{|\ubq|}\bigg[\int_{\Omega_I}\frac{\partial}{\partial{\zeta}}\bigg(\alpha\frac{1}{\bx-\bq_{- I}}+\beta\frac{1}{\bx-\bq_{ I}}\bigg)f(\bx)dV_I(\bx)\\
&-\int_{\partial B_{\epsilon}}\bigg(\alpha\frac{\zeta-\eta}{|\bq_I-\bx|^2}+\beta\frac{\zeta+\eta}{|\bq_{-I}-\bx|^2}\bigg)\overline{d\bx^*}f(\bx)\bigg]\\
=&\frac{q_i}{|\ubq|}\bigg(\int_{\Omega_I}\frac{\partial}{\partial{\zeta}}S^{-1}(\bq,\bx)f(\bx)dV_I(\bx)-\pi(-\alpha \overline{I}f(\bq_I)+\beta \overline{I}f(\bq_{-I}))\bigg)\\
=&\frac{q_i}{|\ubq|}\bigg(\int_{\Omega_I}\frac{\partial}{\partial{\zeta}}S^{-1}(\bq,\bx)f(\bx)dV_I(\bx)-\pi(\alpha I f(\bq_I)-\beta {I}f(\bq_{-I}))\bigg)
\end{align*}
Therefore, we obtain
\begin{align}\label{pqi}
&2\pi\frac{\ubq}{|\ubq|^2}\sum_{i=1}^mq_i\partial_{q_i}T_{\Omega_I}f(\bq)\nonumber\\
=&\frac{\ubq}{|\ubq|^2}\sum_{i=1}^mq_i\frac{q_i}{|\ubq|}\bigg(\int_{\Omega_I}\frac{\partial}{\partial{\zeta}}S^{-1}(\bq,\bx)f(\bx)dV_I(\bx)-\pi(\alpha I f(\bq_I)-\beta {I}f(\bq_{-I}))\bigg)\nonumber\\
=&\int_{\Omega_I}I_{\bq}\frac{\partial}{\partial{\zeta}}S^{-1}(\bq,\bx)f(\bx)dV_I(\bx)-\pi I_{\bq}(\alpha I f(\bq_I)-\beta {I}f(\bq_{-I}))\nonumber\\
=&\int_{\Omega_I}I_{\bq}\frac{\partial}{\partial{\zeta}}S^{-1}(\bq,\bx)f(\bx)dV_I(\bx)+\pi (\alpha  f(\bq_I)+\beta f(\bq_{-I})).
\end{align}
Combining \eqref{pq0} with \eqref{pqi}, we have that
\begin{align*}
&2\pi G_{\bq}T_{\Omega_I}f(\bq)=2\pi \partial_{q_0}T_{\Omega_I}f(\bq)+2\pi \frac{\ubq}{|\ubq|^2}\sum_{i=1}^mq_i\partial_{q_i}T_{\Omega_I}f(\bq)\\
=&\int_{\Omega_I}\bigg(\frac{\partial}{\partial q_0}+I_{\bq}\frac{\partial}{\partial{\zeta}}\bigg)S^{-1}(\bq,\bx)f(\bx)dV_I(\bx)+2\pi (\alpha  f(\bq_I)+\beta f(\bq_{-I}))\\
=&\int_{\Omega_I}G_{\bq}S^{-1}(\bq,\bx)f(\bx)dV_I(\bx)+2\pi (\alpha  f(\bq_I)+\beta f(\bq_{-I}))\\
=&2\pi (\alpha  f(\bq_I)+\beta f(\bq_{-I})).
\end{align*}
Therefore, we obtain
\begin{align*}
G_{\bq}T_{\Omega_I}f(\bq)&=\alpha f(\bq_I)+\beta f(\bq_{-I}).
\end{align*}
Further, if $f\in\mathcal{S}(\Omega_D)$, with Theorem \ref{Rep}, we immediately have $G_{\bq}T_{\Omega_I}f(\bq)=f(\bq)$, which completes the proof.
\end{proof}
Further, we need the following technical lemma to interchange differentiation and integration as follows.
\begin{lemma}\label{Lemma1}
Let $\Omega_D\subset \R_*^{m+1}$ be a  bounded and axially symmetric domain, $f\in C^1(\overline{\Omega_D})\cap\mathcal{S}(\Omega_D)$, then we have
\begin{align*}
\partial_{q_i}T_{\Omega_D}f(\bq)=\frac{2}{\omega_{m-1}}\int_{\mathbb{S}^+}\partial_{q_i}T_{\Omega_I}f(\bq)dS(I).
\end{align*}
\end{lemma}
\begin{proof}
Let $\bx=x_0+\ubx\in\Omega_D$, we rewrite $\ubx=rI$ with $I\in\mathbb{S}$ by spherical coordinates. Then, we have the volume element $dV(\bx)=dx_0dV(\ubx)=r^{m-1}dx_0drdS(I)=r^{m-1}dV_I(\bx)dS(I)$, where $dS(I)$ is the surface element on the half sphere $\mathbb{S}^+$. Hence, we have
\begin{align*}
T_{\Omega_D}f(\bq)=&-\frac{1}{2\pi}\int_{\Omega_D}K(\bq,\bx)f(\bx)dV(\bx)\\
=&\frac{-1}{\pi\omega_{m-1}}\int_{\mathbb{S}^+}\int_{\Omega_I}S^{-1}(\bq,\bx)f(\bx)dV_I(\bx)dS(I)\\
=&\frac{2}{\omega_{m-1}}\int_{\mathbb{S}^+}T_{\Omega_I}f(\bq)dS(I).
\end{align*}
Now, we verify the two conditions in Theorem \ref{inter}. Firstly, the existence of $\partial_{q_i}T_{\Omega_I}f(\bq)$ has already been justified in the theorem above. Secondly, the equations \eqref{TOI} and \eqref{PQI} give us that
\begin{align*}
2\pi \partial_{q_0}T_{\Omega_I}f(\bq)=&\int_{\Omega_I}\bigg(\alpha\frac{q_0-x_0}{|\bq_I-\bx|^2}+\beta\frac{q_0-x_0}{|\bq_{-I}-\bx|^2}\bigg)\big(\partial_{\bx_I}f(\bx)\big)dV_I(\bx)\nonumber\\
&-\int_{\partial\Omega_I}\bigg(\alpha\frac{q_0-x_0}{|\bq_I-\bx|^2}+\beta\frac{q_0-x_0}{|\bq_{-I}-\bx|^2}\bigg)\overline{d\bx^*}f(\bx),\\
2\pi\partial_{q_i}T_{\Omega_I}f(\bq)=&\frac{q_i}{|\ubq|}\bigg[\int_{\Omega_I}\bigg(\alpha\frac{\zeta-\eta}{|\bq_I-\bx|^2}+\beta\frac{\zeta+\eta}{|\bq_{-I}-\bx|^2}\bigg)\big(\partial_{\bx_I}f(\bx)\big)dV_I(\bx)\nonumber\\
&-\int_{\partial \Omega_I}\bigg(\alpha\frac{\zeta-\eta}{|\bq_I-\bx|^2}+\beta\frac{\zeta+\eta}{|\bq_{-I}-\bx|^2}\bigg)\overline{d\bx^*}f(\bx)\bigg].
\end{align*}
One can observe that the integral over $\partial\Omega_I$ is bounded with $f\in C^1(\overline{\Omega_D})$, since the integrand has no singularity on the boundary. Further, from the homogenecity of $\displaystyle\frac{q_0-x_0}{|\bq_{\pm I}-\bx|^2},\ \displaystyle\frac{\zeta\pm\eta}{|\bq_{\pm I}-\bx|^2}$, $f\in C^1(\overline{\Omega_D})$ and a similar argument for \eqref{Homo}, one can easily see that $\partial_{q_j}T_{\Omega_I}f(\bq),\ j=0,\ldots,m$ is integrable over $\mathbb{S}^+$ for all $\bq\in\Omega_D$. Indeed, for $j=0$, we can see that
\begin{align*}
\bv\partial_{q_0}T_{\Omega_I}f(\bq)\bv
\leq& C_1\int_{\Omega_I}\frac{|q_0-x_0|}{|\bq_I-\bx|^2}|\partial_{\bx_I}f(\bx)|dV_I(\bx)\\
&+C_2\int_{\Omega_I}\frac{|q_0-x_0|}{|\bq_{-I}-\bx|^2}|\partial_{\bx_I}f(\bx)|dV_I(\bx)+C'.
\end{align*}
Since the integrals above make sense as Cauchy's principle values, we need to consider, for instance, the limit of the following integral
\begin{align*}
\int_{\Omega_I\backslash B(\bq_I,\epsilon)}\frac{|q_0-x_0|}{|\bq_I-\bx|^2}|\partial_{\bx_I}f(\bx)|dV_I(\bx).
\end{align*}
We notice that $dV_I(\bx)$ includes $|\bx-\bq_I|$ if we apply spherical coordinates. Further, since $f\in C^1(\Omega_D)$ and $\Omega_D$ is bounded and closed, we can choose $R>0$ such that $\Omega_D\subset B(\bq_I,R)$ and we assume that $f(\bx)=0$ when $\bx\in B(\bq_I,R)\backslash\Omega_D$. Then, we set $\bx=\bq_I+s\bo{t}$, where $\bo{t}=t_0+It_1\in\Omega_I$ and we have
\begin{align*}
 &\int_{\Su^+}\int_{\Omega_I\backslash B(\bq_I,\epsilon)}\frac{|q_0-x_0|}{|\bq_I-\bx|^2}|\partial_{\bx_I}f(\bx)|dV_I(\bx)dS(I)\\
&\leq C\int_{\epsilon}^R\int_{|\bo{t}|=1}|t_0|dsdS(\bo{t})dS(I)\leq C,
\end{align*}
where the first inequality above because of $f\in C^1(\overline{\Omega_D})$, and this justifies the second condition in Theorem \ref{inter}, which completes the proof for $j=0$, and a similar argument can be applied for $j=1,2,\ldots,m$.
\end{proof}
Now, we claim that $T_{\Omega_D}$ is a right inverse of $G_{\bq}$ as follows.
\begin{theorem}\label{leftinv}
Let $\Omega_D\subset \R_*^{m+1}$ be a bounded axially symmetric domain and $f\in C^1(\overline{\Omega_D})\cap \mathcal{S}(\Omega_D)$, then
\begin{align*}
G_{\bq}T_{\Omega_D}f(\bq)=f(\bq).
\end{align*}
\end{theorem}
\begin{proof}
Let $\bx=x_0+\ubx\in\Omega_D$, we rewrite $\ubx=rI$ with $I\in\mathbb{S}$ by spherical coordinates. Then, we have the volume element $dV(\bx)=dx_0dV(\ubx)=r^{m-1}dx_0drdS(I)=r^{m-1}dV_I(\bx)dS(I)$, where $dS(I)$ is the surface element on the sphere $\mathbb{S}$. Hence, we have
\begin{align*}
T_{\Omega_D}f(\bq)=&-\frac{1}{2\pi}\int_{\Omega_D}K(\bq,\bx)f(\bx)dV(\bx)\\
=&\frac{-1}{\pi\omega_{m-1}}\int_{\mathbb{S}^+}\int_{\Omega_I}S^{-1}(\bq,\bx)f(\bx)dV_I(\bx)dS(I)\\
=&\frac{2}{\omega_{m-1}}\int_{\mathbb{S}^+}T_{\Omega_I}f(\bq)dS(I).
\end{align*}
With Lemma \ref{Lemma1} and Theorem \ref{thmOI}, we have
\begin{align*}
G_{\bq}T_{\Omega_D}f(\bq)=\frac{2}{\omega_{m-1}}\int_{\mathbb{S}^+}G_{\bq}T_{\Omega_I}f(\bq)dS(I)
=\frac{2}{\omega_{m-1}}\int_{\mathbb{S}^+}f(\bq)dS(I)=f(\bq),
\end{align*}
which completes the proof.
\end{proof}
With the proof of Lemma \ref{Lemma1}, for $f\in C^1(\overline{\Omega_D})\cap \mathcal{S}(\Omega_D)$, we can have that
\begin{align*}
&\partial_{q_0}T_{\Omega_D}f(\bq)\\
=&\frac{1}{2\pi}\int_{\mathbb{S}^+}\bigg[\int_{\Omega_I}\frac{\partial S^{-1}(\bq,\bx)}{\partial q_0}f(\bx)dV_I(\bx)+\pi f(\bq)\bigg]dS(I),\\
&\partial_{q_i}T_{\Omega_D}f(\bq)\\
=&\frac{q_i}{2\pi|\ubq|}\int_{\mathbb{S}^+}\bigg[\int_{\Omega_I}\frac{\partial S^{-1}(\bq,\bx)}{\partial{\zeta}}f(\bx)dV_I(\bx)-\pi [\alpha I f(\bq_I)-\beta {I}f(\bq_{-I})]\bigg]dS(I)\\
=&-\frac{q_i\ubq}{2\pi|\ubq|^2}I_{\bq}\int_{\mathbb{S}^+}\bigg[\int_{\Omega_I}\frac{\partial S^{-1}(\bq,\bx)}{\partial{\zeta}}f(\bx)dV_I(\bx)-\pi [\alpha I f(\bq_I)-\beta {I}f(\bq_{-I})]\bigg]dS(I)\\
=&-\frac{q_i\ubq}{2\pi|\ubq|^2}\int_{\mathbb{S}^+}\bigg[I_{\bq}\int_{\Omega_I}\frac{\partial S^{-1}(\bq,\bx)}{\partial{\zeta}}f(\bx)dV_I(\bx)+\pi f(\bq)\bigg]dS(I),
\end{align*}
where the last equation comes from the fact that
$I_{\bq}\alpha I=-\alpha,\ I_{\bq}\beta I=\beta,$
 $f\in\mathcal{S}(\Omega_D)$ and the representation formula in Theorem \ref{Rep}.
 \par
 Now, we define the $L^p$ space over an axially symmetric domain $\Omega_D$ for slice functions as $$\mathcal{L}^p(\Omega_D)=\mathcal{S}(\Omega_D)\cap L^p(\Omega_D).$$
 Proposition \ref{BoundT} immediately gives us the following mapping property for $T_{\Omega_D}$.
\begin{proposition}\label{p38}
Let $\Omega_D\subset\R^{m+1}_*$ be a bounded domain and $p>\max\{m,2\}$, then we have that
\begin{align*}
T_{\Omega_D}:\ \CL^p(\Omega_D)\longrightarrow \CL^p(\Omega_D)
\end{align*}
is continuous.
\end{proposition}
\begin{proof}
As we mentioned, $T_{\Omega_D}$ maps $L^p(\Omega_D)$ to $L^p(\Omega_D)$ from Proposition \ref{BoundT}. Here, we only prove $T_{\Omega_D}$ also maps $L^{p}( \Omega_{D}) $ to $S( \Omega_D) $. Indeed, since $S^{-1}(\bq,\bx)$ is a slice function with respect to $\boldsymbol{q}$ (\cite{15}), we assume that $S^{-1}( \boldsymbol{q}, \bx) $ can be rewritten as
\begin{align*}
S^{-1}(\boldsymbol{q},\boldsymbol{ x})=F_{1}(z,\boldsymbol{x})+JF_{2}(z,\boldsymbol{x}),
\end{align*}
where $F(z,\bx)=F_1(z,\boldsymbol{x})+iF_2(z,\boldsymbol{x})$ is the stem function which induces the slice Cauchy kernel $S^{-1}(\bq,\bx)$ as in the definition. Hence, we have
\begin{align*}
&T_{\Omega_D}f(\boldsymbol{q})=\int_{\Omega_D}K(\boldsymbol{q},\boldsymbol{x})f(\boldsymbol{x})dV(\boldsymbol{x})=\int_{\Omega_D}\frac{2S^{-1}(\boldsymbol{q},\boldsymbol{x})}{\omega_{m-1}|\underline{\boldsymbol{x}}|^{m-1}}f(\boldsymbol{x})dV(\boldsymbol{x})\\
=&\frac{2}{\omega_{m-1}}\left[\int_{\Omega_D}\frac{F_1(z,\boldsymbol{x})}{|\underline{x}|^{m-1}}f(\boldsymbol{x})dV(\boldsymbol{x})+J\int_{\Omega_D}\frac{F_2(z,\boldsymbol{x})}{|\underline{\boldsymbol{x}}|^{m-1}}f(\boldsymbol{x})dV(\boldsymbol{x})\right].
\end{align*}
If we let
\begin{align*}
&H(z)=H_{1}(z)+iH_{2}(z)\\
=&\left[\frac{1}{2\omega_{m-1}}\int_{\Omega_{D}}\frac{F_{1}(z,\boldsymbol{x})}{|\boldsymbol{x}|^{m-1}}f(\boldsymbol{x})dV(\boldsymbol{x})\right]+i\left[\frac{1}{2\omega_{m-1}}\int_{\Omega_{D}}\frac{F_{2}(z,\boldsymbol{x})}{|\boldsymbol{x}|^{m-1}}f(\boldsymbol{x})dV(\boldsymbol{x})\right],
\end{align*}
then since $F_1(z,\boldsymbol{x}),F_2(z,\boldsymbol{x}),f(\boldsymbol{x})$ are all real Clifford-valued, we know that $H_1(z)$ and $H_2(z)$ are both real Clifford-valued. Further, since $F_1(\bar{z},\boldsymbol{x})=F_1(z,\boldsymbol{x})$ and $F_2(\overline{z},\boldsymbol{x})=-F_2(z,\boldsymbol{x})$, we immediately have $H_1( \overline {z}) = H_1( z) $ and $H_2( \overline {z}) = $ $-H_2(z)$, which tells us that $H(z)$ is a stem function. Therefore, the function $T_{\Omega_D}f$ induced by $H( z) $ is a slice function, which completes the proof.
\end{proof}
\section{Hodge decomposition of the Banach space of slice $L^p$ functions}
Let $\Omega_D\subset\R^{m+1}_*$ be a bounded axially symmetric domain and $1\leq p<\infty$, the norm of a Clifford-valued function $f\in L^p(\Omega_D) $ is given by
\begin{align*}
||f||_{L^p(\Omega_D)}:=\bigg(\int_{\Omega_D}|f(\bx)|^pdV(\bx)\bigg)^{\frac{1}{p}}.
\end{align*}
Notice that the set $\mathcal{A}^p(\Omega_D):=\ker G\cap \mathcal{L}^p(\Omega_D)$, equipped with the norm inherited from $L^p(\Omega_D)$, is called the slice monogenic Bergman space. To prove that $\mathcal{L}^p(\Omega_D)$ and $\mathcal{A}^p(\Omega_D)$ are closed subspaces of $L^p(\Omega_D)$, we introduce a proposition as follows.
\begin{proposition}\cite[Proposition 2]{Berg}\label{P2}
Let $\Omega_D\subset\R^{m+1}_*$ be a bounded axially symmetric domain. For any compact set $K\subset\Omega_D$, there exists a constant $\lambda_K>0$ such that
\begin{align*}
\text{sup}\{|f(\bq)|:\ \bq\in K\}\leq \lambda_K||f||_{L^p},\ \forall f\in\mathcal{A}^p(\Omega_D).
\end{align*}
\end{proposition}
\begin{remark}
Proposition $2$ in \cite{Berg} is stated in the quaternionic case with $p=2$, but one can see that it can be generalized to the higher dimensions in the context of Clifford algebras with $1\leq p<\infty$. Further, Proposition $2$ in \cite{Berg} requires $\Omega$ to be a slice domain, this is because slice regular functions there are defined on slice domains. In our case, slice monogenic functions are induced by holomorphic stem functions or defined as slice functions which can also be annihilated by the differential operator $G$, and in this context, we does not require the domain $\Omega_D$ to be a slice domain.
\end{remark}
\begin{proposition}
 Let $\Omega_D\subset\R^{m+1}_*$ be a bounded axially symmetric domain and $1\leq p<\infty$, then
 \begin{enumerate}
 \item$\mathcal{L}^p(\Omega_D)$ is a closed subspace of $L^p(\Omega_D)$,
 \item$ \mathcal{A}^p(\Omega_D)$ is a closed subspace of $\mathcal{L}^p(\Omega_D)$.
 \end{enumerate}
 \end{proposition}
 \begin{proof}
 \begin{enumerate}
\item Suppose that we have a sequence $f_n\in\mathcal{L}^p(\Omega_D)$, which converges to $f\in L^p( \Omega_D) $ in the $L^p$ norm. To show that $f\in \mathcal{S}(\Omega_D)$, we only need to show that $f$ satisfies the representation formula given in Theorem $2.3$. For any $I\in \mathbb{S} , $we have

\begin{align}\label{eqn41}
&||f(\boldsymbol{x})-\alpha f(\boldsymbol{x}_{I})-\beta f(\boldsymbol{x}_{-I})||_{L^{p}}\nonumber\\
\leq&||f(\boldsymbol{x})-f_{n}(\boldsymbol{x})||_{L^{p}}+||f_{n}(\boldsymbol{x})-\alpha f_{n}(\boldsymbol{x}_{I})-\beta f_{n}(\boldsymbol{x}_{-I})||_{L^{p}}\nonumber\\
&+||\alpha f_{n}(\boldsymbol{x}_{I})+\beta f_{n}(\boldsymbol{x}_{-I})-\alpha f(\boldsymbol{x}_{I})-\beta f(\boldsymbol{x}_{-I})||_{L^{p}}.
\end{align}
According to the assumption for $\{f_n\}$, the first summand goes to zero when $n$ goes to infinity. The second summand above is equal to zero, since $f_n$ is a slice function for all $n$. Now, we estimate the third summand as follows.
\begin{align*}
||\alpha f_{n}(\boldsymbol{x}_{I})+\beta f_{n}(\boldsymbol{x}_{-I})-\alpha f(\boldsymbol{x}_{I})-\beta f(\boldsymbol{x}_{-I})||_{L^{p}}\\\leq||\alpha(f_{n}(\boldsymbol{x}_{I})-f(\boldsymbol{x}_{I}))||_{L^{p}}+||\beta(f_{n}(\boldsymbol{x}_{-I})-f(\boldsymbol{x}_{-I}))||_{L^{p}}.
\end{align*}
Now, we assume $\bx= u+ I_{\boldsymbol{x}}v, $since the sequence $\{ f_n\} $ converges to $f$ in the $L^{p}$ norm, which can be written as
\begin{align}\label{eqn42}
&||f_{n}-f||_{L^{p}}^{p}=\int_{\Omega_{D}}|f_{n}(\boldsymbol{x})-f(\boldsymbol{x})|^{p}dV(\boldsymbol{x})\nonumber\\
=&\int_{\mathbb{S}^{+}}\int_{\Omega_{I_{x}}}|(f_{n}-f)(u+I_{\boldsymbol{x}}v)|^{p}v^{2}dudvdS(I_{\boldsymbol{x}}).
\end{align}
 By assumption, equation \eqref{eqn42} converges to zero when $n$ goes to infinity. This implies that for almost every $I_{\boldsymbol{x}}\in\mathbb{S}^{+}$, we have
\begin{align*}
\int_{\Omega_{I_{\boldsymbol{x}}}}|(f_{n}-f)(u+I_{\boldsymbol{x}}v)|^{p}v^2dudv\longrightarrow 0,\quad\mathrm{(when~}n\to\infty).
\end{align*}
Hence, we have
\begin{align*}
&||\alpha(f_n(\boldsymbol{x}_I)-f(\boldsymbol{x}_I))||^p_{L^p}=\int_{\Omega_D}|\alpha(f_n(\boldsymbol{x}_I)-f(\boldsymbol{x}_I))|^pdV(\boldsymbol{x}) \\
\leq &C\int_{\mathbb{S}^{+}}\int_{\Omega_{I_{x}}}|f_{n}(u+Iv)-f(u+Iv)|^{p}v^{2}dudvdS(I_{x}) \\
=& C\int_{\mathbb{S}^{+}}\int_{\Omega_{I}}|(f_{n}(u+Iv)-f(u+Iv))|^{p}v^2dudvdS(I_{x})\\=&C'\int_{\Omega_{I_{x}}}|f_{n}(u+Iv)-f(u+Iv)|^{p}v^{2}dudv\longrightarrow 0,\ (n\rightarrow \infty),
\end{align*}
where the last second equality comes from the fact that the domains of the variables $u,v$ on $\Omega_I$ and $\Omega_{I_{\bx}}$ are the same, since $\Omega_I$ can be obtained by rotating $\Omega_{I_{\bx}}$ around the real axis. Similarly, we also have $\|\beta(f_n(x_{-I})-f(x_{-I}))\|_{L^p}\longrightarrow 0$ when $n\to\infty$. Thus, with equation \eqref{eqn41}, we have $\|f(x)-\alpha f(x_I)-\beta f(x_{-I})\|_{L^p}\longrightarrow 0$ when $n\to \infty$,  which completes the proof.

\item Here, we adapt the proof given in \cite[Theorem 2]{Berg}. Suppose $\{f_n\}$ is a convergent sequence in $\mathcal{A}^p(\Omega_D)$ and its limit function in $\mathcal{L}^p(\Omega_D)$ is denoted by $f$. Proposition \ref{P2} tells us that there exists a function $g:\Omega_D\longrightarrow\Clm$ given by
\begin{align*}
g(\bq):=\lim_{n\rightarrow\infty}f_n(\bq), \quad \bq\in\Omega_D,
\end{align*}
and $\{f_n\}$ converges uniformly to $g$ on compact subsets of $\Omega_D$, which implies that $g$ is a slice monogenic function on $\Omega_D$. Now, for any compact subset $K\subset\Omega_D$, we have
\begin{align*}
0\leq &\int_K|(f-g)(\bx)|^pdV(\bx)\\
\leq& \int_K|(f-f_n)(\bx)|^pdV(\bx)+\int_K|(f_n-g)(\bx)|^pdV(\bx)\\
\leq& ||f-f_n||_{L^p}^p+\int_K|(f_n-g)(\bx)|^pdV(\bx),
\end{align*}
which goes to zero when $n\rightarrow \infty$. This shows that $f=g\in\mathcal{A}^p(\Omega_D)$, which completes the proof.
\end{enumerate}
 \end{proof}
\begin{remark}
Proposition \ref{eqn42} tells us that $\mathcal{L}^p(\Omega_D)$ and $ \mathcal{A}^p(\Omega_D)$ are both Banach spaces with $1\leq p<\infty$.
 \end{remark}
Now, we introduce a notation
\begin{align*}
S_{\partial\Omega_D}f(\bq)=p.v.\frac{1}{2\pi}\int_{\partial\Omega_D}K(\bq,\boldsymbol{x})n(\boldsymbol{x})f(\boldsymbol{x})d\sigma(\boldsymbol{x}),\  \bq\in\partial\Omega_D,
\end{align*}
where $p.v.$ stands for the principal value in the above singular integral. To make this article self-contained, we introduce a Plemelj formula given in \cite{11} as follows.
\begin{theorem} [Plemelj integral formula]\cite[Theorem 4.2]{11}\label{Pl} Let $\Omega_D\subset \mathbb{R}_* ^{m+ 1}$ be a bounded axially symmetric domain with smooth boundary $\partial\Omega_D$, and $y(t)\in\Omega_D$ is a smooth path in $\mathbb{R}_*^{m+1}$ and it has non-tangential limit $\bq\in\partial\Omega_D$ as $t\rightarrow 0$. Then, for each H\"older continuous slice function $f:\Omega_D\longrightarrow\Clm$ defined on $\Omega_D$, we have
	\begin{align*}
		\lim_{t\to0}F_{\partial\Omega_D}f(y(t))
		=\begin{cases}
			\dfrac{f(\boldsymbol{q})}{2}+S_{\partial\Omega_D}f(\bq),\:y(t)\in\Omega_D;\\
			-\dfrac{f(\boldsymbol{q})}{2}+S_{\partial\Omega_D}f(\bq),\:y(t)\in\mathbb{R}_*^{m+1}\backslash\overline{\Omega_D}.\end{cases}
	\end{align*}
\end{theorem}
With Theorem \ref{Pl}, we can easily have a result on slice monogenic continuation as follows.
\begin{corollary}\label{ext}
Let $\Omega_D\subset \mathbb{R} _*^{m+ 1}$ be a bounded axially symmetric domain with smooth boundary $\partial\Omega_D$. The relation
\begin{align*}
S_{\partial_{\Omega_D}}g(\bq)=\frac{g(\bq)}{2},\ \text{for\ all}\ \bq\in\partial\Omega_D
\end{align*}
is necessary and sufficient so that $g$ represents the boundary values of a slice monogenic function $S_{\partial_{\Omega_D}}g$ defined in $\Omega_D$. On the other hand, the relation
\begin{align*}
S_{\partial_{\Omega_D}}g(\bq)=-\frac{g(\bq)}{2},\ \text{for\ all}\ \bq\in\partial\Omega_D
\end{align*}
is necessary and sufficient so that $g$ represents the boundary values of a slice monogenic function $S_{\partial_{\Omega_D}}g$ defined in $\R^{m+1}_*\backslash\Omega_D$.
\end{corollary}
\begin{proof}
Let $f$ be the slice monogenic continuation into the domain $\Omega_D$ of the function $g$ given on $\partial\Omega_D$. Then, the Cauchy integral formula gives us that $f(\bq)=F_{\partial\Omega_D}g(\bq)$. Therefore, the non-tangential boundary values of $f$ are $g$. Using the Plemelj formula in Theorem \ref{Pl}, we get
\begin{align*}
g(\bq)=\frac{g(\bq)}{2}+S_{\partial\Omega_D}g(\bq), \ \text{for all}\  \bq\in\partial\Omega_D,
\end{align*}
which leads to
\begin{align*}
S_{\partial_{\Omega_D}}g(\bq)=\frac{g(\bq)}{2},\ \text{for all}\  \bq\in\partial\Omega_D.
\end{align*}
If vice versa, we have
\begin{align*}
S_{\partial_{\Omega_D}}g(\bq)=\frac{g(\bq)}{2},\ \text{for all}\  \bq\in\partial\Omega_D.
\end{align*}
Then, the Plemelj formula tells us that $F_{\partial\Omega_D}g(\bq)$ has the boundary value $g$, therefore, it is the slice monogenic continuation of $g$ into $\Omega_D$. The proof for the exterior domain case can be obtained similarly.
\end{proof}
\begin{remark}
It is worth pointing out that a Plemelj formula is given in \cite[Theorem 2.3]{18} for $\frac{\partial}{\partial\bx^c}$, which is equivalent to the Plemelj formula introduced in Theorem \ref{Pl}, when we consider slice functions defined on $\Omega_D\subset\R^{m+1}_*$.
\end{remark}
 Following the idea of the proof in \cite[Theorem 8.7]{23}, we introduce a Hodge decomposition of $\mathcal{L}^p(\Omega_D)$ as follows.
\begin{theorem}[Hodge decomposition]\label{Hodge}
Let $\Omega_D\subset \mathbb{R}_* ^{m+ 1}$ be a bounded axially symmetric domain and $p>\max\{m,2\}$. Then, the space $\mathcal{L}^p(\Omega_D)$ allows the orthogonal decomposition
\begin{align*}
\mathcal{L}^{p}(\Omega_{D})=\mathcal{A}^{p}(\Omega_{D})\oplus \bigg(|\ux|^{1-m}G \CL^p_0(\Omega_D)\bigg)
\end{align*}
with respect to the Clifford-valued inner product  given by
\begin{align*}
 \langle f,g\rangle:=\int_{\Omega_D}\overline{f(\bx)}g(\bx)dV(\bx),\quad f,g\in {L}^{p}(\Omega_{D}).
 \end{align*}
\end{theorem}
\begin{proof}
Let $Y= \mathcal{L} ^p( \Omega_D) \ominus\mathcal{A}^{p}(\Omega_{D})$ be the orthogonal complement to the space $\mathcal{A}^{p}(\Omega_{D})$ with respect to the inner product  $\langle\cdot,\cdot\rangle$ given above. For any $f\in Y$ with $p>\max\{m,2\}$, we have $|\ux|^{m-1}f\in L^p(\Omega_D)$, so that $g= T_{\Omega_D}(|\ux|^{m-1}f)( \bx) \in L^{p}(\Omega_D)$. Then, we obtain $f(\bx)=\dfrac{Gg(\bx)}{|\ux|^{m-1}}$, and for any $\phi\in \mathcal{A}^{p}(\Omega_{D})$, we have
\begin{align*}
0=\int_{\Omega_{D}}\overline{\phi(\boldsymbol{x})}f(\boldsymbol{x})dV(\boldsymbol{x})=\int_{\Omega_{D}}\overline{\phi(\boldsymbol{x})}\dfrac{Gg(\bx)}{|\ux|^{m-1}}dV(\boldsymbol{x}).
\end{align*}
In particular, let $\phi_l(\bx)=\dfrac2{\omega_{m-1}}\overline{S^{-1}(\bq_l,\bx)}$, where $\{\bq_l\}$ is dense in $\mathbb{R}_*^{m+1}\backslash\overline{\Omega_D}.$ Obviously, we have $\overline{\phi_l(\bx)}|\underline{\bx}|^{1-m}=K(\bq_l,\bx)$, $\phi_l(\bx)\in\mathcal{L}^p(\Omega_D)$ and $\phi_l( \boldsymbol{x}) G_{\boldsymbol{x}}= 0, $ where $G_{\boldsymbol{x}}$ means that $G$ is a differential operator with respect to $\bx$. Then, we have
\begin{align*}
&0=\int_{\Omega_{D}}\overline{\phi_{l}(\boldsymbol{x})}\dfrac{Gg(\bx)}{|\ux|^{m-1}}dV(\boldsymbol{x})=\frac{2}{\omega_{m-1}}\int_{\mathbb{S}^+}\int_{\Omega_{I}}S^{-1}(\boldsymbol{q}_l,\boldsymbol{x})(Gg(\boldsymbol{x}))dV_I(\boldsymbol{x}) dS(I)\\
=&\frac{-2}{\omega_{m-1}}\int_{\mathbb{S}^+}\bigg[\int_{\Omega_{I}}(S^{-1}(\boldsymbol{q}_l,\boldsymbol{x})G)g(\boldsymbol{x})dV_I(\boldsymbol{x})-\int_{\partial\Omega_{I}}S^{-1}(\boldsymbol{q}_l,\boldsymbol{x})\boldsymbol{n}(\boldsymbol{x})g(\boldsymbol{x})d\sigma_I(\boldsymbol{x}) \bigg]dS(I)\\
=&\frac{2}{\omega_{m-1}}\int_{\mathbb{S}^+}\int_{\partial\Omega_{I}}S^{-1}(\boldsymbol{q}_l,\boldsymbol{x})\boldsymbol{n}(\boldsymbol{x})g(\boldsymbol{x})d\sigma_I(\boldsymbol{x}) dS(I)\\
=&\int_{\partial\Omega_{D}}K(\boldsymbol{q}_l,\boldsymbol{x})\boldsymbol{n}(\boldsymbol{x})g(\boldsymbol{x})d\sigma(\boldsymbol{x})=F_{\partial\Omega_{D}}(\mathrm{tr}g)(\boldsymbol{q}_l),
\end{align*}
 where the last second equation comes from the fact that $d\sigma(\boldsymbol{x})=|\underline{\bx}|^{m-1}d\sigma_I(\boldsymbol{x})$ and $\text{tr}g$ denotes the trace of $g$. Therefore, we have $F_{\partial\Omega_D}( $tr$g) = 0$ in $\mathbb{R} _* ^{m+1}\backslash\overline {\Omega_D}$ for continuity. Then, the Plemelj formula in Theorem \ref{Pl} tells us that
\begin{align*}
S_{\partial\Omega_D}\text{tr}g(\bq)=\frac{\text{tr}g(\bq)}{2},\ \text{for all}\ \bq\in\partial\Omega_D.
 \end{align*}
 Therefore, with Corollary \ref{ext}, the trace $\text{tr}g$ can be slice monogenicly extended into the domain $\Omega_D$. Here, we denote the continuation by $h$. Then, we have $\text{tr}_{\partial\Omega_D}g=\text{tr}_{\partial\Omega_D}h$ and the trace operator $\text{tr}_{\partial\Omega_D}$ describes the restriction onto the boundary $\partial\Omega_D.$ Now, we denote $\omega:=g-h$, and obviously we have $\text{tr}_{\partial\Omega_{D}}\omega= 0$ and we get $\omega\in  L^p_0(\Omega_D).$ Next, we prove that
$G\omega=Gg=|\ux|^{m-1}f\in \mathcal{S}(\Omega_D)$. Indeed, since $f\in \mathcal{S}(\Omega_D)$, we suppose $f$ is induced by the stem function $F(z)=F_1(z)+iF_2(z)$, where $F_1,F_2$ satisfy the even-odd conditions given in \eqref{evenodd}. One can easily check that the functions $H_1$ and $H_2$, defined by \begin{align*}
H_1(z):=\bigg\vert\frac{z-\overline{z}}{2i}\bigg\vert^{m-1}F_1(z),\ H_2(z):=\bigg\vert\frac{z-\overline{z}}{2i}\bigg\vert^{m-1}F_2(z),
\end{align*}
also satisfy the even-odd conditions. Further, $|\ux|^{m-1}f$ is induced by the stem function $H(z)=H_1(z)+iH_2(z)$, which justifies that $G\omega=Gg=|\ux|^{m-1}f\in \mathcal{S}(\Omega_D)$ and this completes the proof.
\end{proof}
\begin{remark}
It is worth pointing out that the condition $p>\max\{m,2\}$ in the theorem above guarantees that $|\ux|^{1-m}G \CL^p_0(\Omega_D)\subset L^p(\Omega_D)$ with a similar argument as in Proposition \ref{BoundT}. Further, for the second component in the decomposition, there are no differentiability problems for $G \CL^p_0(\Omega_D)$. More specifically, in the proof, we already pointed out this component consists of $T_{\Omega_D}(|\ux|^{m-1}f)( \bx)$, which is differentiable given in Lemma \ref{Lemma1}, and $h$, which is also differentiable because of the slice monogenic extension.
\end{remark}
 The Hodge decomposition gives rise to two orthogonal projections on the corresponding subspaces $\mathcal{A}^{p}(\Omega_{D})$ and $Y$, i.e., we have
\begin{align*}
&\bo{P}:\:\mathcal{L}^{p}(\Omega_{D})\longrightarrow\mathcal{A}^{p}(\Omega_{D}),\\
&\bo{Q}:\:\mathcal{L}^{p}(\Omega_{D})\longrightarrow |\ux|^{1-m}G \CL^p_0(\Omega_D).
\end{align*}
Moreover, the projection $\bo{P}$ can be considered as a generalization of the classical Bergman projection in the slice regular function theory. Next, we give a connection between the image of the operator $\bo{Q}$ and the boundary value of $T_{\Omega_D}|\ux|^{m-1}f.$
 \begin{proposition}
Let $\Omega_D\subset \mathbb{R}_* ^{m+ 1}$ be a bounded axially symmetric domain. A slice function $f:\Omega_D\longrightarrow\Clm$ belongs to $\text{im}\boldsymbol{Q}$ if and only if $\text{tr}_{\partial\Omega_D}T_{\Omega_D}|\ux|^{m-1}f=0$, where $\text{im}\boldsymbol{Q}$ is the image of the operator $\boldsymbol{Q}$ and $\text{tr}_{\partial\Omega_D}f$ is the restriction of $f$ onto $\partial\Omega_{D}$.
\end{proposition}
\begin{proof}
The proof is similar as the argument given in \cite[Proposition 8.9]{23}. Firstly, let $f\in \text{im}\boldsymbol{Q}$, then there exists a function $g\in L_0^{p}(\Omega_D)$, such that $f=|\ux|^{1-m}Gg$. Then, the Borel-Pompeiu formula given in Theorem \ref{BPF} tells us that
\begin{align*}
T_{\Omega_{D}}|\ux|^{m-1}f=T(Gg)=g-F_{\partial\Omega_{D}}g=g,
\end{align*}
therefore, $\mathrm{tr}_{\partial\Omega_{D}}T_{\Omega_{D}}|\ux|^{m-1}f= \mathrm{tr}_{\partial\Omega_{D}}g= 0.$
\par
Vice versa, we assume that $\mathrm{tr}_{\partial\Omega_D}T_{\Omega_D}|\ux|^{m-1}f= 0, $ which is the same as
\begin{align*}
\mathrm{tr}_{\partial\Omega_{D}}T_{\Omega_{D}}\boldsymbol{P}|\ux|^{m-1}f+\mathrm{tr}_{\partial\Omega_{D}}T_{\Omega_{D}}\boldsymbol{Q}|\ux|^{m-1}f=0.
\end{align*}
Obviously, $\text{tr}_{\partial\Omega_D}T_{\Omega_D}\boldsymbol{Q}|\ux|^{m-1}f=0$, so we have $\text{tr}_{\partial\Omega_D}T_{\Omega_D}\boldsymbol{P}|\ux|^{m-1}f= 0$. Further, Proposition \ref{p38} tells us that $T_{\Omega_D}\boldsymbol{P}|\ux|^{m-1}f\in L^{p}(\Omega_D).$ Hence, we have $T_{\Omega_D}\boldsymbol{P}|\ux|^{m-1}f\in L_{0}^{p}(\Omega_{D}).$ Then, we can get $G( T_{\Omega_{D}}\bo{P}|\ux|^{m-1}f) = \bo{P}|\ux|^{m-1}f\in \text{im}\bo{Q}.$ However, we obviously have $G( T_{\Omega_{D}}\bo{P}|\ux|^{m-1}f) = \bo{P}|\ux|^{m-1}{f} \in \text{im}\bo{P}$, and $\text{im}\bo{P}\cap \text{im}\bo{Q}= 0$  leads to $\bo{P}|\ux|^{m-1}f= 0$ for all $\bx\in\Omega_D\subset\mathbb{R}^{m+1}_*$.  This implies that $|\ux|^{m-1}f\in\text{im}\bo{Q}$, which is equivalent to $f\in\text{im}\bo{Q}$.
\end{proof}

\begin{remark}
One can notice that the discussions above require the domain $\Omega_D\subset\R_*^{m+1}$. In other words, $\Omega_D$ does not intersect the real line $\R$. As we mentioned in the very beginning, this restriction comes from the definition of $G$, which requires $|\underline{\bx}|\neq 0$. We notice that the works done in \cite{18,GP} suggest that our discussions can possibly be generalized to the case of domains intersecting the real line with the slice derivative operator $\frac{\partial}{\partial \bx^c}$. However, it is not a straightforward generalization to the case of domains intersecting the real line, and we will discuss it in an upcoming paper.
\end{remark}
\bmhead{Acknowledgements}
Chao Ding is supported by the National Natural Science Foundation of China (No. 12271001), the Natural Science Foundation of Anhui Province (No. 2308085MA03), and the Excellent University Research and Innovation Team in Anhui Province (No. 2024AH010002). Zhenghua Xu is supported by the Natural Science Foundation of Anhui Province (No. 2308085MA04).
\bmhead{Statements and Declarations}
No potential conflict of interest was reported by the authors. No datasets were generated or analyzed during the current study.


\begin{thebibliography}{1}

\bibitem{ACS}\textsc{D. Alpay, F. Colombo, I. Sabadini}, \textit{Quaternionic de Branges spaces and characteristic operator function}, SpringerBriefs in Mathematics, Springer, Cham, 2020.

\bibitem{B1}\textsc{H. Begehr, G. N. Hile}, Nonlinear Riemann boundary value problems for a nonlinear elliptic system in the plane. \textit{Math. Z.}, \textbf{179} (1982), 241--261.

\bibitem{B2} \textsc{H. Begehr, G. N. Hile}, Riemann boundary value problems for a nonlinear elliptic system. \textit{Complex Variables Theory Appl.}, \textbf{1}(1982/83), 239--261.


\bibitem{Bo}\textsc{B. Bojarski}, Subsonic flow of compressible fluid. In: Mathematical problems in fluid mechanics. \textit{Warszawa: Polish Acad. Sci.}, (1967), 9--32.

\bibitem{PU}\textsc{P. Cerejeiras, U. K\"ahler}, On Beltrami Equations in Clifford Analysis and Its Quasi-conformal Solutions, in \textit{Clifford Analysis and Its Applications}. NATO Science Series, vol 25. Springer, Dordrecht, 2001.


\bibitem{CG}\textsc{F. Colombo, J. Gantner}, \textit{Quaternionic closed operators, fractional powers and fractional diffusion processes}, Operator Theory: Advances and Applications, 274. Birkh\"auser/Springer, Cham, 2019.

\bibitem{CGK}\textsc{F. Colombo, J. Gantner, D. P. Kimsey}, \textit{Spectral theory on the S-spectrum for quaternionic operators}. Operator Theory: Advances and Applications, 270. Birkh\"auser/Springer, Cham, 2018.


\bibitem{Berg}\textsc{F. Colombo, J. O. Gonz\'alez-Cervantes, M. E. Luna-Elizarrar\'as, I. Sabadini, M. Shapiro}, On Two Approaches to the Bergman Theory for Slice Regular Functions, in \textit{Advances in Hypercomplex Analysis}, Springer INdAM Series, vol 1. Springer, Milano, 2013.


\bibitem{4}\textsc{F. Colombo, J. O. Gonz\'alez-Cervantes, I. Sabadini}, A non constant coefficients differential operator associated to slice monogenic functions, \textit{Trans. Am. Math. Soc.}, \textbf{365}(2013), 303--318.

\bibitem{CS1}\textsc{F. Colombo, I. Sabadini}, On some properties of the quaternionic functional calculus, \textit{J. Geom. Anal.}, \textbf{19(3)}(2009),601--627.

\bibitem{CS2}\textsc{F. Colombo, I. Sabadini}, On the formulations of the quaternionic functional calculus, \textit{J. Geom. Phys.}, \textbf{60(10)}(2010), 1490--1508.

\bibitem{6}\textsc{F. Colombo, I. Sabadini, D. C. Struppa}, Slice monogenic functions, \textit{Israel J. Math.}, \textbf{171}(2009),   385--403.

\bibitem{7}\textsc{F. Colombo, I. Sabadini, D. C. Struppa}, \textit{Noncommutative Functional Calculus, Theory and Applications of Slice Hyperholomorphic Functions}, Progress in Mathematics 289, Birkh\"auser, 2011.


\bibitem{9}\textsc{C. G. Cullen}, An integral theorem for analytic intrinsic functions on quaternions, \textit{Duke Math. J.}, \textbf{32}(1965), 139--148.

\bibitem{11}\textsc{C. Ding, X.  Cheng}, Integral formulas for slice Cauchy-Riemann operator and applications, \textit{Adv. Appl. Clifford Algebras}, \textbf{34}(2024), article no. 32, 16 pp.

\bibitem{12}\textsc{G. Folland}, \textit{Real Analysis: Modern Techniques and Their Applications}, 2nd edition, Wiley, 2007.

\bibitem{13}\textsc{G. Gentili, C. Stoppato, D. C. Struppa}, \textit{Regular Functions of a Quaternionic Variable (2nd ed.)}, Springer Cham, 2022.

\bibitem{19}\textsc{G. Gentili, D. C. Struppa}, A new approach to Cullen-regular functions of a quaternionic variable, \textit{C. R. Math. Acad. Sci. Paris.}, \textbf{342}(2006), 741--744.

\bibitem{20}\textsc{G. Gentili, D. C. Struppa}, A new theory of regular function of a quaternionic variable, \textit{Adv. Math.}, \textbf{216}(2007), 279--301.
\bibitem{14}\textsc{R. Ghiloni}, Slice Fueter-regular functions, \textit{J. Geom. Anal.}, \textbf{31}(2021), 11988--12033.

\bibitem{GMP}\textsc{R. Ghiloni, V. Moretti, A. Perotti}, \textit{Continuous slice functional calculus in quaternionic Hilbert spaces}, \textit{Rev. Math. Phys.}, \textbf{25}(2013),  1350006, 83 pp.

\bibitem{GMP1}\textsc{R. Ghiloni, V. Moretti, A. Perotti}, \textit{Spectral representations of normal operators in quaternionic Hilbert spaces via intertwining quaternionic PVMs}, \textit{Rev. Math. Phys.}, \textbf{29}(2017), 1750034, 73 pp.

\bibitem{15}\textsc{R. Ghiloni, A. Perotti}, Slice regular functions on real alternative algebras, \textit{Adv. Math.}, \textbf{226}(2011), 1662--1691.

\bibitem{18}\textsc{R. Ghiloni, A. Perotti}, Volume Cauchy formulas for slice functions on real associative *-algebras, \textit{Complex Var. Elliptic Equ.}, \textbf{58}(2013), 1701--1714.

\bibitem{GP} \textsc{R. Ghiloni, A. Perotti}, Global differential equations for slice regular functions. \textit{Math. Nachr.}, \textbf{287}(2014), 561--573.



\bibitem{22}\textsc{J. O. Gonz\'alez-Cervantes, D. Gonz\'alez-Campos}, The global Borel-Pompeiu-type formula for quaternionic slice regular functions, \textit{Complex Var. Elliptic Equ.}, \textbf{66}(2021), 721--730.

\bibitem{GK}\textsc{K. G\"urlebeck, U. K\"ahler}, On a spatial generalization of the complex $\Pi$-Operator, \textit{Z. Anal. Anwend.}, \textbf{15}(1996), 283--297.

\bibitem{Gur}\textsc{K. G\"urlebeck, U. K\"ahler, M. Shapiro}, On the $\Pi$-operator in hyperholomorphic function theory, \textit{Adv. Appl. Clifford Algebras}, \textbf{9}(1999), 23--40.

\bibitem{23}\textsc{K. G\"urlebeck, K. Habetha, W. Spr\"o\ss ig}, \textit{Holomorphic functions in the plane and $n$-dimensional space}, Birkh\"auser Verlag, Basel, 2008.

\bibitem{26}\textsc{M. Jin, G. Ren, I. Sabadini}, Slice Dirac operator over octonions, \textit{Israel J. Math.}, \textbf{240}(2020), 315--344.


\bibitem{Kah}\textsc{U. K\"ahler}, On quaternionic Beltrami Equations, \textit{Clifford Algebras and their Applications in Mathematical Physics}. Progress in Physics, vol 19. Birkh\"auser, Boston, 2000.



\bibitem{Vekua}\textsc{I. N. Vekua}, Generalized Analytic Functions. Oxford: Pergamon Press 1962.

\bibitem{Haiyan}\textsc{H.  Wang, X. Bian}, The right inverse of Dirac operator in octonionic space, \textit{J. Geom. Phys.}, \textbf{119}(2017),139--145.

\bibitem{B3}\textsc{G.  Wen, H. Begehr}, \textit{Boundary Value Problems for Elliptic Equations and Systems}, Pitman Monogr. Surveys Pure Appl. Math. 46. Harlow: Longman 1990.

    \bibitem{XS}\textsc{Z. Xu, I. Sabadini}, Generalized partial-slice monogenic functions: a synthesis of two function theories, \textit{Adv. Appl. Clifford Algebras}, \textbf{34}(2024), Paper No. 10, 14 pp.



\end{thebibliography}

\end{document}